\documentclass[11pt]{article}

\usepackage{amsmath}
\usepackage{amssymb}
\usepackage{amsthm}
\usepackage{mathtools}
\usepackage{tabu}
\usepackage{textcomp}
\usepackage{listings}
\usepackage{graphicx}
\usepackage{tikz}
\usepackage{arydshln}
\usepackage[margin=1in]{geometry}
\usepackage{indentfirst}
\usepackage{array}
\usepackage{xltabular}
\usepackage{caption}
\usepackage{chessfss}

\newtheorem{proposition}{Proposition}[section]

\newtheorem{definition}{Definition}[section]

\setboardfontsize{15pt}
\newcommand{\wking}{\WhiteKingOnWhite}
\newcommand{\bking}{\BlackKingOnWhite}

\newcommand{\varT}{\mathcal{T}}
\newcommand{\varU}{\mathcal{U}}

\newcommand{\tikzpic}[2]{
    \begin{tikzpicture}[scale=#1]
        #2
    \end{tikzpicture}}

\newcommand{\ctikzpic}[2]{
    \begin{center}
        \begin{tikzpicture}[scale=#1]
            #2
        \end{tikzpicture}
    \end{center}}

\newcommand{\drawtile}[2]{
    \draw (#1,#2) rectangle (#1+1,#2+1)}

\newcommand{\drawOiv}[2]{
    \filldraw[fill=white] (#1-1,#2+1)--(#1-1,#2-1)--(#1+1,#2-1)--(#1+1,#2+1)--(#1-1,#2+1);
    \draw[dashed] (#1-1,#2)--(#1+1,#2); \draw[dashed] (#1,#2-1)--(#1,#2+1)
}

\newcommand{\Ii}{
    \tikzpic{0.25}{
        \drawtile{0}{0};
    }
}

\newcommand{\Oiv}{\!\!\!\!\raisebox{-0.25em}{
    \tikzpic{0.25}{
        \drawOiv{0}{0};
    }}
}

\newcommand{\drawwtile}[6]{
    \draw (#1,#2) rectangle (#1+1,#2+1);
    \node at(#1+0.25,#2+0.5) [left] {#3};
    \node at(#1+0.5,#2+0.75) [above] {\phantom{$d$} #4 \phantom{$p$}};
    \node at(#1+0.5,#2+0.25) [below] {\phantom{$d$} #5 \phantom{$p$}};
    \node at(#1+0.75,#2+0.5) [right] {#6}
}

\newcommand{\wtile}[4]{\hspace{-0.9em}
    \raisebox{-0.3cm-1em}{
        \tikzpic{0.6}{
            \drawwtile{0}{0}{$#1$}{$#2$}{$#3$}{$#4$};
        }
    }\hspace{-0.6em}
}

\newcommand{\wtilek}[4]{\hspace{-0.9em}
    \raisebox{-0.3cm-1em}{
        \tikzpic{0.6}{
            \drawwtile{0}{0}{$#1$}{$#2$}{$#3$}{$#4$};
            \node at(0.5,0.5) {\wking};
        }
    }\hspace{-0.6em}
}

\newcommand{\wtilebk}[4]{\hspace{-0.9em}
    \raisebox{-0.3cm-1em}{
        \tikzpic{0.6}{
            \drawwtile{0}{0}{$#1$}{$#2$}{$#3$}{$#4$};
            \node at(0.5,0.5) {\bking};
        }
    }\hspace{-0.6em}
}

\newcommand{\wIii}[6]{\hspace{-0.9em}
    \raisebox{-0.3cm-1em}{
        \tikzpic{0.6}{
            \drawwtile{0}{0}{$#1$}{$#2$}{$#4$}{};
            \drawwtile{1}{0}{}{$#3$}{$#5$}{$#6$};
        }
    }\hspace{-0.6em}
}

\newcommand{\wIiiki}[6]{\hspace{-0.9em}
    \raisebox{-0.3cm-1em}{
        \tikzpic{0.6}{
            \drawwtile{0}{0}{$#1$}{$#2$}{$#4$}{};
            \drawwtile{1}{0}{}{$#3$}{$#5$}{$#6$};
            \node at(0.5,0.5){\wking};
        }
    }\hspace{-0.6em}
}

\newcommand{\wIiikii}[6]{\hspace{-0.9em}
    \raisebox{-0.3cm-1em}{
        \tikzpic{0.6}{
            \drawwtile{0}{0}{$#1$}{$#2$}{$#4$}{};
            \drawwtile{1}{0}{}{$#3$}{$#5$}{$#6$};
            \node at(1.5,0.5){\wking};
        }
    }\hspace{-0.6em}
}

\newcommand{\wIiii}[8]{\hspace{-0.9em}
    \raisebox{-0.1cm-1em}{
        \tikzpic{0.2}{
            \drawwtile{0}{0}{$#1$}{$#2$}{$#5$}{};
            \drawwtile{1}{0}{}{$#3$}{$#6$}{};
            \drawwtile{2}{0}{}{$#4$}{$#7$}{$#8$};
        }
    }\hspace{-0.6em}
}

\title{Independent Set Enumeration in King Graphs by Tensor Network Contractions}
\author{Kai Liang}
\date{\today} 

\begin{document}
    \maketitle

    \begin{abstract}
        This paper discusses the enumeration of independent sets in king graphs of size $m \times n$, based on the tensor network contractions algorithm given in reference~\cite{tilEnum}.
        We transform the problem into Wang tiling enumeration within an $(m+1) \times (n+1)$ rectangle and compute the results for all cases where $m + n \leq 79$ using tensor network contraction algorithm,
        and provided an approximation for larger $m, n$.

        Using the same algorithm, we also enumerated independent sets with vertex number restrictions.
        Based on the results, we analyzed the vertex number that maximize the enumeration for each pair $(m, n)$.
        Additionally, we compute the corresponding weighted enumeration, where each independent set is weighted by the number of its vertices (i.e., the total sum of vertices over all independent sets).
        The approximations for larger $m, n$ are given as well.

        Our results have added thousands of new items to the OEIS sequences A089980 and A193580.
        In addition, the combinatorial problems above are closely related to the hard-core model in physics.
        We estimate some important constants based on the existing results, and our estimation of the entropy constant gives more five exact digits to the existing results of OEIS A247413.
        \end{abstract}

    \small
    \begin{center}
        \textbf{Note}
    \end{center}
    ~~~~This is a reversion of the original paper (which can also be found on this website).
    The changes mainly include:

    (1) Changed the data of $c_{m,n}$, the vertex number that maximum the enumeration.
    These data were found to be incorrect for unknown reasons.

    (2) Deleted the subsequent discussion of using the estimation formula to predict $c_{m,n}$ .
    Its accuracy was found to be an illusion caused by the estimation method of related constants.

    (3) Applied more careful numerical analysis methods to the estimation of the constants, and more accurate estimation is given.

    (4) Changed the term ``entropy constant'' to ``average free energy''.
    The definition in this paper is consistent with the latter, while the former is generally defined as the exponential of the latter.

    (5) Corrected some typos (like ``convex'' to ``vertex'').
    \normalsize

    \section{Preliminary}

    We begin by introducing the definitions of king graphs independent sets in graphs.

    \begin{definition}
        (King graph) An $m \times n$ \textit{king graph} $K_{m\times n}$ is a graph whose vertex set corresponds to an $m \times n$ grid, but edges additionally include diagonal adjacencies, modeling the movement of a chess king on a chessboard.
    \end{definition}

    \begin{definition}
        (Independent Set) Given a graph $G = (V, E)$, an independent set is a subset $S \subseteq V$ such that no two vertices in $S$ are adjacent, i.e., $\forall v_1, v_2 \in S$, $\{v_1, v_2\} \notin E$.
    \end{definition}

    Thus, an independent set of the $m\times n$ king graph with $c$ vertices can be seen as placing $c$ kings on a $m\times n$ chessboard so that they cannot attack each other.
    We denote the number of independent sets (including the empty set) on an $m \times n$ king graph as $N_{m\times n}$.

    It can be observed that $N_{m\times n}$ is equal to the number of ways to arrange non-overlapping $2 \times 2$ square tiles $\Oiv$ within an $(m+1) \times (n+1)$ rectangular grid~\cite{Enum, Enum2}.
    This correspondence arises as follows:
    Consider the $(m+1) \times (n+1)$ grid, treating each position not in the leftmost column or top row as a vertex,
    each \textit{bottom-right corner} of a $\Oiv$ as a king.
    A valid tiling (with no overlapping tiles) ensures that the bottom-right corners of all tiles form an independent set in the king graph, as adjacency (including diagonal) would imply overlapping tiles.

    The following figures illustrates a $4 \times 5$ king graph, and it's corresponding $\Oiv$ arrangement.
    \begin{center}
        \tikzpic{0.6}{
            \draw (0,0) grid (3,4);
            \draw (0,3)--(1,4); \draw (0,2)--(2,4); \draw (0,1)--(3,4); \draw (0,0)--(3,3); \draw (1,0)--(3,2); \draw (2,0)--(3,1);
            \draw (3,3)--(2,4); \draw (3,2)--(1,4); \draw (3,1)--(0,4); \draw (3,0)--(0,3); \draw (2,0)--(0,2); \draw (1,0)--(0,1);
            \node at(0,0) {\wking};\filldraw[fill=white] (0,1) circle (0.1);\filldraw[fill=white] (0,2) circle (0.1);\node at(0,3) {\wking};\filldraw[fill=white] (0,4) circle (0.1);
            \filldraw[fill=white] (1,0) circle (0.1);\filldraw[fill=white] (1,1) circle (0.1);\filldraw[fill=white] (1,2) circle (0.1);\filldraw[fill=white] (1,3) circle (0.1);\filldraw[fill=white] (1,4) circle (0.1);
            \filldraw[fill=white] (2,0) circle (0.1);\filldraw[fill=white] (2,1) circle (0.1);\node at(2,2) {\wking};\filldraw[fill=white] (2,3) circle (0.1);\filldraw[fill=white] (2,4) circle (0.1);
            \node at(3,0) {\wking};\filldraw[fill=white] (3,1) circle (0.1);\filldraw[fill=white] (3,2) circle (0.1);\filldraw[fill=white] (3,3) circle (0.1);\node at(3,4) {\wking};
        }
        \raisebox{1.5cm}{$\Rightarrow~~$}
        \tikzpic{0.6}{
            \filldraw[fill=gray!20] (0-1.5,0-0.5) rectangle (3+0.5, 4+1.5);

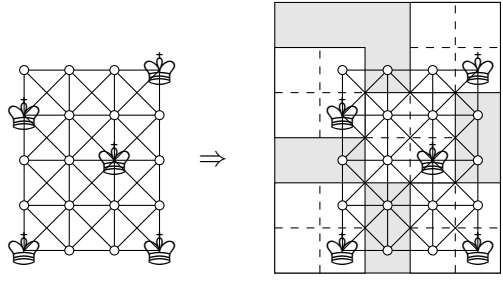
            \drawOiv{0-0.5}{0+0.5};\drawOiv{3-0.5}{0+0.5};\drawOiv{2-0.5}{2+0.5};\drawOiv{0-0.5}{3+0.5};\drawOiv{3-0.5}{4+0.5};

            \draw (0,0) grid (3,4);
            \draw (0,3)--(1,4); \draw (0,2)--(2,4); \draw (0,1)--(3,4); \draw (0,0)--(3,3); \draw (1,0)--(3,2); \draw (2,0)--(3,1);
            \draw (3,3)--(2,4); \draw (3,2)--(1,4); \draw (3,1)--(0,4); \draw (3,0)--(0,3); \draw (2,0)--(0,2); \draw (1,0)--(0,1);

            \node at(0,0) {\wking};\filldraw[fill=white] (0,1) circle (0.1);\filldraw[fill=white] (0,2) circle (0.1);\node at(0,3) {\wking};\filldraw[fill=white] (0,4) circle (0.1);
            \filldraw[fill=white] (1,0) circle (0.1);\filldraw[fill=white] (1,1) circle (0.1);\filldraw[fill=white] (1,2) circle (0.1);\filldraw[fill=white] (1,3) circle (0.1);\filldraw[fill=white] (1,4) circle (0.1);
            \filldraw[fill=white] (2,0) circle (0.1);\filldraw[fill=white] (2,1) circle (0.1);\node at(2,2) {\wking};\filldraw[fill=white] (2,3) circle (0.1);\filldraw[fill=white] (2,4) circle (0.1);
            \node at(3,0) {\wking};\filldraw[fill=white] (3,1) circle (0.1);\filldraw[fill=white] (3,2) circle (0.1);\filldraw[fill=white] (3,3) circle (0.1);\node at(3,4) {\wking};
        }
    \end{center}
    
    \section{Independent set enumeration}

    \subsection{Algorithms}

    Following the method in~\cite{tilEnum}, we can transform the enumeration of such tile arrangement into an equivalence Wang tiling enumeration within rectangular regions.

    Clearly, the arrangement of such tiles is equivalent to the tiling of $\Ii+\Oiv$ covering the entire chessboard, by using the unit tile $\Ii$ to fill all positions without $\Oiv$.
    The corresponding Wang tiles are (the following tile edges without characters are treated as having the boundary character $\sharp$):

    \[
        \varT = \wtilek{1}{1}{}{}+\wtile{1}{}{1}{}+\wtile{}{}{}{1}+\wtile{}{}{}{},
    \]

    Here, $\wtile{1}{1}{}{}, \wtile{1}{}{1}{}$ and $\wtile{}{}{}{1}$ can \textit{uniquely} combine to form a $\Oiv$ (requiring only $\sharp$ is on the boundary) as shown below:

    \ctikzpic{0.6}{
        \drawOiv{0}{0}; \node at(-0.2, -0.5) {$1$}; \node at(-0.2, 0.5) {$1$}; \node at(0.3, 0) {$1$};
        \node at(0.5,-0.5) {$\wking$};
    }
    while each $\wtile{}{}{}{}$ (e.g. $\wtile{\sharp}{\sharp}{\sharp}{\sharp}$) gives a $\Ii$.
    The horizontal and vertical characters are $\varTheta=\varSigma=\{\sharp, 1\}$, with sizes $\theta=\sigma=2$.

    Therefore, we can apply the algorithm proposed in ~\cite{tilEnum} to compute the enumeration.
    Here, the maximum size of the state tensor (which can be reached at most of the positions) is given by
    \[
        \theta\sigma_0\sigma^{m-1}=2^{m},
    \]  
    where $m$ is the width of the tiling region, $\sigma_0=1$ is the size of vertical alphabet of $\varT_0=\{\sharp\}$,
    obtained by constraining the left-edge character in $\varT$ to be the boundary character $\sharp$, which is the tile set for the leftmost column.

    To reduce memory usage (and thus extend computable widths), we \textit{merge Wang tiles horizontally}. 
    Specifically, Let $\varT^{(l)}$ denote the merged tile set obtained by \textit{horizontally combining} $l$ \textit{tiles} from the original set $\varT$, which also required that the characters on the spliced edges are the same.
    For example,
    \begin{align*}
        \varT^{(1)} &= \varT = \wtilek{1}{1}{}{}+\wtile{1}{}{1}{}+\wtile{}{}{}{1}+\wtile{}{}{}{}; \\
        \varT^{(2)} &= \wIii{}{}{}{}{}{}+\wIii{}{}{}{}{}{1}+\wIiikii{}{}{1}{}{}{}+\wIii{}{}{}{}{1}{}+\wIiiki{1}{1}{}{}{}{}+\wIii{1}{}{}{1}{}{}; \\
        \varT^{(3)} &= \wIiii{}{}{}{}{}{}{}{}+\wIiii{}{}{}{}{}{}{}{1}+\wIiii{}{}{}{1}{}{}{}{}+\wIiii{}{}{}{}{}{}{1}{}
        +\wIiii{}{}{1}{}{}{}{}{}+\wIiii{}{}{1}{}{}{}{}{1}+\wIiii{}{}{}{}{}{1}{}{}+\wIiii{}{}{}{}{}{1}{}{1}\\
                   +&\wIiii{1}{1}{}{}{}{}{}{}+\wIiii{1}{1}{}{}{}{}{}{1}+\wIiii{1}{}{}{}{1}{}{}{}+\wIiii{1}{}{}{}{1}{}{}{1}
        +\wIiii{1}{1}{}{1}{}{}{}{}+\wIiii{1}{}{}{}{1}{}{1}{}+\wIiii{1}{1}{}{}{}{}{1}{}+\wIiii{1}{}{}{1}{1}{}{}{}; \\
        \varT^{(4)} &= \ldots.
    \end{align*}

    We denote the vertical alphabet for $\varT^{(l)}$ as $\varTheta^{(l)}$, and its size as $\theta^{(l)}$.
    When $l\geq 2$, $\theta^{(l)}$ is \textit{smaller} than $\theta^l$.
    For example,
    \[
        \begin{cases}
            \varTheta^{(2)}=\{\sharp\sharp,\sharp1,1\sharp\},
            & \theta^{(2)}=3<4; \\
            \varTheta^{(3)}=\{\sharp\sharp\sharp,\sharp\sharp1,\sharp1\sharp,1\sharp\sharp,1\sharp1\},
            & \theta^{(3)}=5<8; \\
            \varTheta^{(4)}=\{\sharp\sharp\sharp\sharp,\sharp\sharp\sharp1,\sharp\sharp1\sharp,\sharp1\sharp\sharp,1\sharp\sharp\sharp,
            \sharp1\sharp1,1\sharp\sharp1,1\sharp1\sharp\},
            & \theta^{(4)}=8<16.
        \end{cases}
    \]

    This is because certain horizontal character sequences (e.g., the sequences include $11$) lead to invalid tilings and thus do not appear,
    In fact, it's easy to prove that this makes the $\theta^{(l)}$ here the $(l+2)$-th Fibonacci number.
    This reduction in vertical alphabet size enables us to reduce the size of the state tensor through such tile merging.
    However, this will simultaneously increase the size of the transfer tensor, potentially reducing the algorithm's efficiency.
    Therefore, the number of merged tiles (here, 4) should also not be too large.

    When constraining the left-edge characters to be the boundary character $\sharp$, the Wang tile sets above gives
    \begin{align*}
        \varT^{(1)}_0 &= \wtile{}{}{}{1}+\wtile{}{}{}{}; \\
        \varT^{(2)}_0 &= \wIii{}{}{}{}{}{}+\wIii{}{}{}{}{}{1}+\wIiikii{}{}{1}{}{}{}+\wIii{}{}{}{}{1}{}; \\
        \varT^{(3)}_0 &= \wIiii{}{}{}{}{}{}{}{}+\wIiii{}{}{}{}{}{}{}{1}+\wIiii{}{}{}{1}{}{}{}{}+\wIiii{}{}{}{}{}{}{1}{}
        +\wIiii{}{}{1}{}{}{}{}{}+\wIiii{}{}{1}{}{}{}{}{1}+\wIiii{}{}{}{}{}{1}{}{}+\wIiii{}{}{}{}{}{1}{}{1};\\
        \varT^{(4)}_0 &= \ldots.
    \end{align*}

    Let $m = 4m_1 + m_0$, where $1 \leq m_0 \leq 4$.
    Instead of processing $m$ tiles individually, we replace each row of $m$ $\varT$ with
    a merged tile set $\varT_0^{(m_0)}$ and $m_1$ merged tile sets $\varT^{(4)}$.
    Now the size of the state tensor is
    \[ 
        \begin{cases}
            2 \times 5 \times (\theta^{(4)})^{m_1-1}=10\times8^{m_1-1},& m=4m_1; \\
            2 \times 1 \times (\theta^{(4)})^{m_1}  =2 \times 8^{m_1}, & m=4m_1+1; \\
            2 \times 2 \times (\theta^{(4)})^{m_1}  =4 \times 8^{m_1}, & m=4m_1+2; \\
            2 \times 3 \times (\theta^{(4)})^{m_1}  =6 \times 8^{m_1}, & m=4m_1+3.
        \end{cases}
    \]

    For the case of $m=39$, the original size of the state tensor is $2^{39}\approx5.50\times10^{11}$.
    After such tile merging, the size can be greatly reduced to $6\times 8^{9}\approx8.05\times10^{8}$, about $0.15\%$ of the original.
    This allows us to calculate more results under the same memory limit.
    
    \subsection{Results}
    
    We computed all cases where $m \leq 39$ and $m + n \leq 81$.
    By utilizing transpose symmetry, we got the results of $N_{m\times n}$ for all $m+n\leq 79$, which have been recorded in the OEIS sequence A245013.
    The table below shows the results of $N_{m\times n}$ in the widest cases, $m = 39, n\leq42$:
    \tiny
    \begin{align*}
        n~~ & N_{39\times n}\\
        1~~ & 165580141\\
        2~~ & 733007751851\\
        3~~ & 502179158360159299\\
        4~~ & 22771944702872450167161\\
        5~~ & 4140603472432724183285215455\\
        6~~ & 361264217123294193531786810807269\\
        7~~ & 46171388216658592659967120655934594283\\
        8~~ & 4829005428999699203482047276834319460871985\\
        9~~ & 560823794943799981958596276541103531065285978787\\
        10~~ & 61653632830449131139551142952360173537910722679009015\\
        \ldots & \ldots\\
        \shortstack[l]{37\\\phantom{0}}~~ &
        \shortstack[l]{1331464027881280234024031532809705979471117841830664993119506315222058844040398017114991077499161882\\719134524891948563581295826734084626060455864192492513080227177353211012535812522762875881}\\
        \shortstack[l]{38\\\phantom{0}}~~ &
        \shortstack[l]{1491579151481223081575668823347931866528869620731772167833793863638288790141439807559317683861174272\\66540898642136794886735068023332063753540430618542683985290363866576627704123956874510769675783}\\
        \shortstack[l]{39\\\phantom{0}}~~ &
        \shortstack[l]{1670948909947452458237211642998627476178642949226304933158196184708941334570072518008961410076515171\\0544643590249894397326545482379274046055465455795978525196665171332045124603707542958232439062741757}\\
        \shortstack[l]{40\\\phantom{0}\\\phantom{0}}~~ &
        \shortstack[l]{1871888767214833519778416569659361340742641936841298481720379826733827750270553548107482491956907465\\5474839871780014244550497037156737408868058735630687755030675037063429735777160848520237745717661442\\57325}\\
        \shortstack[l]{41\\\phantom{0}\\\phantom{0}}~~ &
        \shortstack[l]{2096992634754029797182005348934572385984645467541110974296241944237840648768502058893972758349436005\\5500934087992366486581710388641344268328582197341083490099652633556625253693536995761571600008977796\\9484415685}\\
        \shortstack[l]{42\\\phantom{0}\\\phantom{0}}~~ &
        \shortstack[l]{2349166353779710880226003473508935706680803713992984214900651174502649190135094980604793655680345991\\1651504352540487291021069136729539404762109765347692971125898749468839489811972274400159593061719075\\979938123357909}
    \end{align*}
    \normalsize
    As can be observed, when the width $m$ is fixed, as the height $n$ increases, the enumeration trend approximates exponential growth.

    In fact, independent sets on the king graph can be viewed as a specific \textit{hard-core model}~\cite{lim2,entropy}.
    We can regard each king as a particle, and the limit of no attack between Kings is that particles do not overlap.
    Similar with the other hard-core models~\cite{const, const2, lim, entropy2}, there exists a constant $\kappa$ such that:
    \begin{align*}
        \kappa_{m\times n}&\coloneq\frac{\ln N_{m\times n}}{mn}; \\
        \kappa &\coloneq \lim_{m,n\rightarrow\infty}\kappa_{m\times n},
    \end{align*}
    which is called the \textit{average free energy} of the model.
    Here $\kappa$ is the average entropy contributed by each vertex, so the total entropy of the system (i.e. the natural logarithm of the number of states) is $mn\kappa$.
    In addition, $\kappa$ is the entropy of the infinite king graph.

    The line charts below illustrate the values of $\kappa_{m\times n}$ derived from our results, where $m=n, 2n,3n$.
    The horizontal axis are $n$ and $\frac{m+n}{mn}$, respectively.
    \begin{center}
        % n axis
        \tikzpic{1}{
            \draw[->](0,0)--(5,0);\node at(5,0)[right]{$n$};
            \draw[step=1.25] (0,-0.1)grid(4.5,0);\node at(0,-0.1)[below]{0};\node at(1.25,-0.1)[below]{10};\node at(2.5,-0.1)[below]{20};\node at(3.75,-0.1)[below]{30};
            \draw[->](0,0)--(0,5.5);\node at(0,5.5) [above] {$\kappa_{m\times n}$};
            \draw (-0.1,0)grid(0,5.5);\node at(-0.1,0)[left]{0.28};\node at(-0.1,1)[left]{0.33};\node at(-0.1,2)[left]{0.38};\node at(-0.1,3)[left]{0.43};\node at(-0.1,4)[left]{0.48};\node at(-0.1,5)[left]{0.53};
            \draw (0.250,2.447)--(0.375,2.300)--(0.500,1.586)--(0.625,1.414)--(0.750,1.189)--(0.875,1.077)--(1.000,0.971)--(1.125,0.898)--(1.250,0.836)--(1.375,0.787)--(1.500,0.745)--(1.625,0.711)--(1.750,0.681)--(1.875,0.655)--(2.000,0.632)--(2.125,0.612)--(2.250,0.594)--(2.375,0.578)--(2.500,0.564)--(2.625,0.551)--(2.750,0.539)--(2.875,0.528)--(3.000,0.519)--(3.125,0.509)--(3.250,0.501)--(3.375,0.493)--(3.500,0.486)--(3.625,0.480)--(3.750,0.473)--(3.875,0.467)--(4.000,0.462)--(4.125,0.457)--(4.250,0.452)--(4.375,0.447)--(4.500,0.443)--(4.625,0.439)--(4.750,0.435)--(4.875,0.431);
            \draw[densely dashed] (0.125,5.386)--(0.250,2.011)--(0.375,1.751)--(0.500,1.275)--(0.625,1.125)--(0.750,0.966)--(0.875,0.878)--(1.000,0.801)--(1.125,0.746)--(1.250,0.700)--(1.375,0.663)--(1.500,0.632)--(1.625,0.606)--(1.750,0.583)--(1.875,0.564)--(2.000,0.547)--(2.125,0.532)--(2.250,0.519)--(2.375,0.507)--(2.500,0.496)--(2.625,0.486)--(2.750,0.477)--(2.875,0.469)--(3.000,0.462)--(3.125,0.455)--(3.250,0.449);
            \draw[densely dotted] (0.125,5.129)--(0.250,1.804)--(0.375,1.603)--(0.500,1.160)--(0.625,1.034)--(0.750,0.889)--(0.875,0.813)--(1.000,0.744)--(1.125,0.695)--(1.250,0.654)--(1.375,0.622)--(1.500,0.594)--(1.625,0.571)--(1.750,0.551)--(1.875,0.534)--(2.000,0.518)--(2.125,0.505)--(2.250,0.493)--(2.375,0.483);

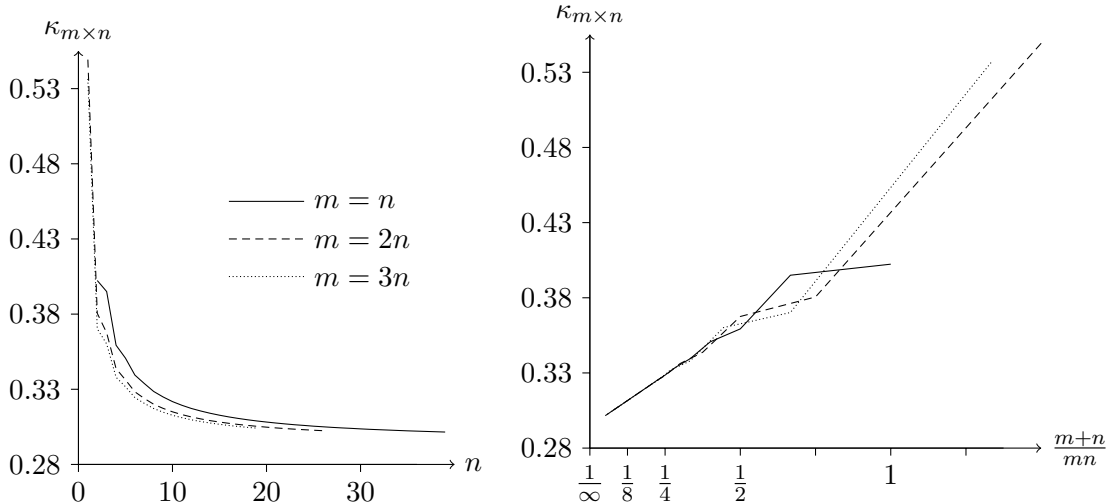
            \draw (2,3.5)--(3,3.5);\node at(3,3.5) [right] {$m=n$};
            \draw[densely dashed] (2,3.0)--(3,3.0);\node at(3,3.0) [right] {$m=2n$};
            \draw[densely dotted] (2,2.5)--(3,2.5);\node at(3,2.5) [right] {$m=3n$};
        }
        % (m+n)/mn axis
        \tikzpic{1}{
            \draw[->](0,0)--(6,0);\node at(6,0)[right]{$\frac{m+n}{mn}$};
            \draw (0,-0.1)grid(5.5,0);\draw (0.5,-0.1)--(0.5,0);\node at(0,-0.1)[below]{$\frac{1}{\infty}$};\node at(0.5,-0.1)[below]{$\frac{1}{8}$};\node at(1,-0.1)[below]{$\frac{1}{4}$};\node at(2,-0.1)[below]{$\frac{1}{2}$};\node at(4,-0.1)[below]{1};
            \draw[->](0,0)--(0,5.5);\node at(0,5.5) [above] {$\kappa_{m\times n}$};
            \draw (-0.1,0)grid(0,5.5);\node at(-0.1,0)[left]{0.28};\node at(-0.1,1)[left]{0.33};\node at(-0.1,2)[left]{0.38};\node at(-0.1,3)[left]{0.43};\node at(-0.1,4)[left]{0.48};\node at(-0.1,5)[left]{0.53};
            \draw (4.000,2.447)--(2.666,2.300)--(2.000,1.586)--(1.600,1.414)--(1.333,1.189)--(1.142,1.077)--(1.000,0.971)--(0.888,0.898)--(0.800,0.836)--(0.727,0.787)--(0.666,0.745)--(0.615,0.711)--(0.571,0.681)--(0.533,0.655)--(0.500,0.632)--(0.470,0.612)--(0.444,0.594)--(0.421,0.578)--(0.400,0.564)--(0.381,0.551)--(0.363,0.539)--(0.347,0.528)--(0.333,0.519)--(0.320,0.509)--(0.307,0.501)--(0.296,0.493)--(0.285,0.486)--(0.275,0.480)--(0.266,0.473)--(0.258,0.467)--(0.250,0.462)--(0.242,0.457)--(0.235,0.452)--(0.228,0.447)--(0.222,0.443)--(0.216,0.439)--(0.210,0.435)--(0.205,0.431);
            \draw[densely dashed] (6.000,5.386)--(3.000,2.011)--(2.000,1.751)--(1.500,1.275)--(1.200,1.125)--(1.000,0.966)--(0.857,0.878)--(0.750,0.801)--(0.666,0.746)--(0.600,0.700)--(0.545,0.663)--(0.500,0.632)--(0.461,0.606)--(0.428,0.583)--(0.400,0.564)--(0.375,0.547)--(0.352,0.532)--(0.333,0.519)--(0.315,0.507)--(0.300,0.496)--(0.285,0.486)--(0.272,0.477)--(0.260,0.469)--(0.250,0.462)--(0.240,0.455)--(0.230,0.449);
            \draw[densely dotted] (5.333,5.129)--(2.666,1.804)--(1.777,1.603)--(1.333,1.160)--(1.066,1.034)--(0.888,0.889)--(0.761,0.813)--(0.666,0.744)--(0.592,0.695)--(0.533,0.654)--(0.484,0.622)--(0.444,0.594)--(0.410,0.571)--(0.381,0.551)--(0.355,0.534)--(0.333,0.518)--(0.313,0.505)--(0.296,0.493)--(0.280,0.483);
        }
    \end{center}

    \subsection{Estimation of constants}

    The charts clearly show that for small $m,n$, the values of $\kappa_{m\times n}$ exhibit significant parity dependence (indicating that the enumeration $N_{m\times n}$ is similarly affected).
    This can be explained through the underlying tile-packing interpretation: since the tile edge length is even, when $n$ is odd, tiles arranged within an even-sized square lattice, which allow more space-efficient configurations compared to odd-sized squares.
    However, as $n$ increases, the contribution of such parity-dependent configurations becomes negligible relative to the total enumeration, causing the parity effect to diminish.

    In addition, when $n$ is sufficiently large, $\kappa_{m\times n}$ converge to $\kappa$ with a deviation approximately proportional to $\frac{m+n}{mn}$.
    Although it is difficult to give a strict proof here, we can give a physical explanation:

    An infinite king graph can be partitioned into infinite $m \times n$ finite king graph,
    so when discussing the enumeration contributed by each vertex, the infinite graph can be interpreted as a finite $m \times n$ graph with \textit{stricter boundary conditions}, i.e., not conflict with the independent set in the adjacent $m\times n$ king graph,
    which reduced the contribution by the vertices on the boundaries.

    Therefore, for sufficiently large $m,  n$, the {boundary effects} occupy an area fraction roughly proportional to $\frac{m+n}{mn}$, which is the ``density'' of the boundaries.
    Consequently, the difference between $\kappa$ and $\kappa_{m\times n}$ due to these changes in boundary constraints are also approximately proportional to $\frac{m+n}{mn}$.
    We denote the ratio by
    \[
        k_{\kappa}\coloneq \lim_{m,n\rightarrow\infty} \frac{\kappa_{m\times n}-\kappa}{\frac{m+n}{mn}} =\lim_{m,n\rightarrow\infty} \frac{(\kappa_{m\times n}-\kappa)mn}{m+n},
    \]
    which gives the approximation of $\kappa_{m\times n}$ and $\ln N_{m\times n}$ when $m,n\rightarrow\infty$:
    \begin{align}\label{eq:kappa_mn}
        \kappa_{m\times n} &=\kappa+k_{\kappa}\cdot\frac{m+n}{mn};       \\
        \ln N_{m\times n}  &= \kappa mn+k_{\kappa}(m+n).
    \end{align}

    According to the According to the law observed above, we can use the existing results ($m\leq39, n\leq42$) to conduct numerical analysis on $\kappa$ and $k$, and get (the $\lessapprox$ symbol marks the boundary that is likely to be closer):
    \begin{align*}
        0.29464076781~5620 &\lessapprox\kappa< 0.29464076781~6145;\\
        0.13549180267~0344 &\lessapprox k     <0.13549180267~3746,
    \end{align*}
    and
    \[
        1.34264395112~3897 \lessapprox e^\kappa< 1.34264395112~4602.
    \]

    The OEIS sequence A247413 provides a value with 10 significant digits of $\kappa$ (calculated with another algorithm in~\cite{const}):
    \begin{equation}\label{eq:kappa}
        e^\kappa = 1.342643951124\ldots \Leftrightarrow \kappa =0.29464076781\ldots.
    \end{equation}
    which is consistent with our results.

    Note that
    \[
        \kappa_{m\times n} =\kappa+k_{\kappa}\cdot\frac{m+n}{mn} +o\left(\frac{m+n}{mn}\right),
    \]
    the error of estimation comes from the term $o\left(\frac{m+n}{mn}\right)$, which is approximately of the same order as $\frac{1}{mn}$, so we can assert that the estimation obtained from the case with larger area $mn$ is more accurate.
    Considering that the space complexity of the algorithm is $\Theta(\theta\sigma_0\sigma^{m-1})$ (and the time complexity is also close to this), that is, the cost of increasing the height $n$ is less than that of increasing the width $m$, we can calculate the case of larger but narrower graphs.
    Here we calculate all the cases of $m\leq34$ and $n\leq100$, and the more accurate estimation is:
    \begin{align*}
        0.294640767816144918~4082&\lessapprox \kappa  <0.294640767816144918~2295;\\
        1.34264395112460129~7851 &\lessapprox e^\kappa<1.34264395112460129~8092 ;\\
        0.1354918026~4626        &\lessapprox k       < 0.1354918026~7004.
    \end{align*}
    Here the estimation of $e^\kappa$ has five more accurate digits than the results in OEIS.

    \section{Vertex number constraints}

    In this section, we compute the enumeration of independent sets under specific vertex number constraints, which equivalently corresponds to the tiling enumeration with restricted numbers of tiles.
    We denote the enumeration of $c$-vertex independent set of an $m\times n$ king graph as $N_{m\times n,c}$.

    \subsection{Algorithm and results}

    Following the algorithm for constrained tiling enumeration described in~\cite{tilEnum}, and applying the tile merging method described in Section 2, we calculated all cases where
    $\lfloor\frac{m+1}{2}\rfloor\lfloor\frac{n+1}{2}\rfloor\leq169$.
    Note that in these cases, the number of vertices is always not greater than 169.

    To reduce the size of the state tensor, we first perform a tiling enumeration with an upper limit of $85$ tiles (discarding cases exceeding $85$ tiles), followed by a tiling enumeration modulo $85$.
    The difference between these two results then gives the enumeration for cases with more than $85$ tiles.

    The table below presents the results of $N_{26\times26,c}$:

    \tiny
    \begin{align*}
        n~~ & N_{m\times n}\\
        1~~ & 1 \\
        2~~ & 676\\
        3~~ & 225600\\
        4~~ & 49553600\\
        5~~ & 8058725534\\
        6~~ & 1034902881936\\
        \ldots&\ldots\\
        96~~ & 28167386774073832059978745704665724536952074690238647572849341190907778280319273024141184\\
        97~~ & 29595870699623036422198914900016026761782408470673221589812956666929724458502743420886580\\
        98~~ & 29841471230559576064030516178272642412720344089273150439553722830764798702419808371257688\\
        99~~ & 28866914666754411802224111001202298696438795180050524431479458758076253380023710258504844\\
        100~~ & 26782702015370312876160118605959589033352968502027811607600959373040503048063529925561244\\
        \ldots&\ldots\\
        166~~ & 51111041090904147713880489165887946\\
        167~~ & 320182849662618325911532557334084\\
        168~~ & 1304024035729212605788041968570\\
        169~~ & 2588716234142991968960920692
    \end{align*}
    \normalsize
    As observed, the enumeration initially increases and then decreases as the vertex number $c$ grows, reaching its maximum at $c = 98$.
    Moreover, the growth rate (i.e., the finite difference) exhibits monotonic decay.

    The chart below presents the results of $N_{26\times26,c},N_{25\times26,c},N_{25\times25,c} $, and provides a clearer illustration of the growth trend by taking the natural logarithm:

    \ctikzpic{0.05}{
        \draw[->] (0,0)--(180,0);\node at(180,0) [right] {$c$};\draw[step=10] (0,0) grid (175,-1);
        \draw[->] (0,0)--(0,100);\draw[step=10] (0,0) grid (-1,95);

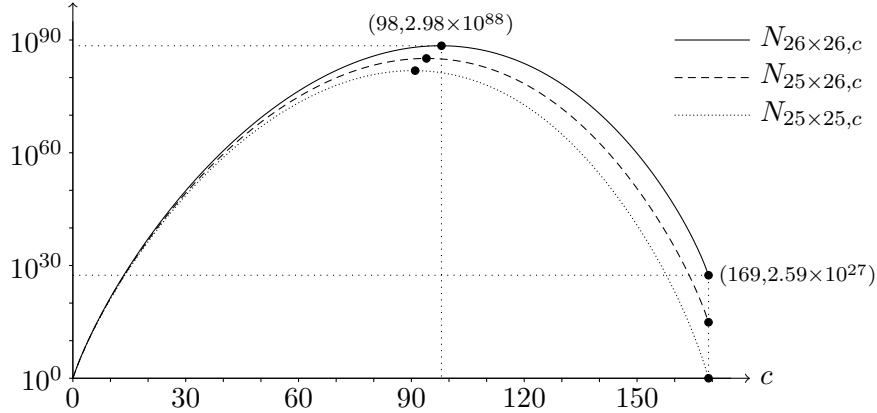
        \node at (0,0) [below] {0};\node at (30,0) [below] {30};\node at (60,0) [below] {60};\node at (90,0) [below] {90};
        \node at (120,0) [below] {120};\node at (150,0) [below] {150};
        \draw[dotted](169,0)--(169,27.41);\draw[dotted](98,88.47)--(98,0);
        \draw[dotted](0,27.41)--(169,27.41);\draw[dotted](98,88.47)--(0,88.47);
        \node at (0,0) [left] {$10^{0}$};\node at (0,30) [left] {$10^{30}$};\node at (0,60) [left] {$10^{60}$};\node at (0,90) [left] {$10^{90}$};
        \node at(98,88.47) [above] {${}_{(98,2.98\times 10^{88})}$};\filldraw (98,88.47) circle (1);
        \node at(169,27.41) [right] {${}_{(169,2.59\times 10^{27})}$};\filldraw (169,27.41) circle (1);
        \draw (0,0.0)--(1,2.83)--(2,5.35)--(3,7.7)--(4,9.91)--(5,12.01)--(6,14.04)--(7,15.99)--(8,17.88)--(9,19.71)--(10,21.49)--(11,23.22)--(12,24.91)--(13,26.55)--(14,28.16)--(15,29.73)--(16,31.27)--(17,32.78)--(18,34.25)--(19,35.7)--(20,37.11)--(21,38.5)--(22,39.87)--(23,41.2)--(24,42.51)--(25,43.8)--(26,45.06)--(27,46.3)--(28,47.52)--(29,48.72)--(30,49.89)--(31,51.05)--(32,52.18)--(33,53.29)--(34,54.38)--(35,55.45)--(36,56.51)--(37,57.54)--(38,58.55)--(39,59.55)--(40,60.53)--(41,61.48)--(42,62.43)--(43,63.35)--(44,64.25)--(45,65.14)--(46,66.01)--(47,66.86)--(48,67.7)--(49,68.52)--(50,69.32)--(51,70.11)--(52,70.87)--(53,71.63)--(54,72.36)--(55,73.08)--(56,73.78)--(57,74.47)--(58,75.13)--(59,75.79)--(60,76.42)--(61,77.04)--(62,77.65)--(63,78.24)--(64,78.81)--(65,79.36)--(66,79.9)--(67,80.43)--(68,80.93)--(69,81.42)--(70,81.9)--(71,82.36)--(72,82.8)--(73,83.22)--(74,83.63)--(75,84.03)--(76,84.41)--(77,84.77)--(78,85.11)--(79,85.44)--(80,85.75)--(81,86.05)--(82,86.33)--(83,86.59)--(84,86.83)--(85,87.06)--(86,87.27)--(87,87.47)--(88,87.65)--(89,87.81)--(90,87.95)--(91,88.08)--(92,88.19)--(93,88.28)--(94,88.35)--(95,88.41)--(96,88.45)--(97,88.47)--(98,88.47)--(99,88.46)--(100,88.43)--(101,88.38)--(102,88.31)--(103,88.22)--(104,88.11)--(105,87.99)--(106,87.84)--(107,87.68)--(108,87.5)--(109,87.3)--(110,87.08)--(111,86.84)--(112,86.58)--(113,86.3)--(114,85.99)--(115,85.67)--(116,85.33)--(117,84.97)--(118,84.59)--(119,84.18)--(120,83.76)--(121,83.31)--(122,82.84)--(123,82.35)--(124,81.84)--(125,81.3)--(126,80.74)--(127,80.16)--(128,79.56)--(129,78.93)--(130,78.28)--(131,77.61)--(132,76.91)--(133,76.19)--(134,75.44)--(135,74.67)--(136,73.87)--(137,73.05)--(138,72.2)--(139,71.33)--(140,70.42)--(141,69.5)--(142,68.54)--(143,67.56)--(144,66.54)--(145,65.5)--(146,64.43)--(147,63.33)--(148,62.2)--(149,61.03)--(150,59.84)--(151,58.61)--(152,57.34)--(153,56.04)--(154,54.7)--(155,53.33)--(156,51.91)--(157,50.45)--(158,48.94)--(159,47.39)--(160,45.78)--(161,44.12)--(162,42.4)--(163,40.61)--(164,38.74)--(165,36.78)--(166,34.71)--(167,32.51)--(168,30.12)--(169,27.41);

        \draw[densely dashed](0,0.0)--(1,2.81)--(2,5.32)--(3,7.64)--(4,9.84)--(5,11.93)--(6,13.93)--(7,15.87)--(8,17.73)--(9,19.55)--(10,21.31)--(11,23.02)--(12,24.69)--(13,26.31)--(14,27.9)--(15,29.45)--(16,30.97)--(17,32.46)--(18,33.91)--(19,35.33)--(20,36.73)--(21,38.1)--(22,39.44)--(23,40.75)--(24,42.04)--(25,43.3)--(26,44.54)--(27,45.76)--(28,46.95)--(29,48.12)--(30,49.27)--(31,50.4)--(32,51.51)--(33,52.59)--(34,53.66)--(35,54.7)--(36,55.73)--(37,56.73)--(38,57.72)--(39,58.69)--(40,59.64)--(41,60.57)--(42,61.48)--(43,62.37)--(44,63.25)--(45,64.1)--(46,64.94)--(47,65.77)--(48,66.57)--(49,67.36)--(50,68.12)--(51,68.88)--(52,69.61)--(53,70.33)--(54,71.03)--(55,71.71)--(56,72.38)--(57,73.03)--(58,73.66)--(59,74.27)--(60,74.87)--(61,75.45)--(62,76.02)--(63,76.57)--(64,77.1)--(65,77.61)--(66,78.11)--(67,78.59)--(68,79.06)--(69,79.5)--(70,79.93)--(71,80.35)--(72,80.75)--(73,81.13)--(74,81.49)--(75,81.83)--(76,82.16)--(77,82.48)--(78,82.77)--(79,83.05)--(80,83.31)--(81,83.55)--(82,83.78)--(83,83.99)--(84,84.18)--(85,84.35)--(86,84.5)--(87,84.64)--(88,84.76)--(89,84.86)--(90,84.95)--(91,85.01)--(92,85.06)--(93,85.09)--(94,85.1)--(95,85.09)--(96,85.06)--(97,85.02)--(98,84.95)--(99,84.87)--(100,84.76)--(101,84.64)--(102,84.5)--(103,84.33)--(104,84.15)--(105,83.95)--(106,83.73)--(107,83.48)--(108,83.22)--(109,82.93)--(110,82.63)--(111,82.3)--(112,81.95)--(113,81.58)--(114,81.19)--(115,80.78)--(116,80.34)--(117,79.89)--(118,79.41)--(119,78.9)--(120,78.38)--(121,77.83)--(122,77.25)--(123,76.66)--(124,76.04)--(125,75.39)--(126,74.72)--(127,74.03)--(128,73.31)--(129,72.56)--(130,71.79)--(131,70.99)--(132,70.17)--(133,69.32)--(134,68.44)--(135,67.54)--(136,66.61)--(137,65.65)--(138,64.66)--(139,63.65)--(140,62.6)--(141,61.53)--(142,60.43)--(143,59.29)--(144,58.13)--(145,56.94)--(146,55.71)--(147,54.45)--(148,53.16)--(149,51.84)--(150,50.49)--(151,49.1)--(152,47.67)--(153,46.21)--(154,44.71)--(155,43.17)--(156,41.59)--(157,39.97)--(158,38.3)--(159,36.58)--(160,34.81)--(161,32.99)--(162,31.1)--(163,29.14)--(164,27.1)--(165,24.97)--(166,22.73)--(167,20.35)--(168,17.78)--(169,14.9);
        \filldraw (94,85.1) circle (1);\filldraw (169,14.9) circle (1);

        \draw[densely dotted](0,0.0)--(1,2.8)--(2,5.28)--(3,7.59)--(4,9.77)--(5,11.84)--(6,13.83)--(7,15.74)--(8,17.59)--(9,19.38)--(10,21.12)--(11,22.82)--(12,24.47)--(13,26.07)--(14,27.64)--(15,29.17)--(16,30.67)--(17,32.13)--(18,33.57)--(19,34.97)--(20,36.34)--(21,37.69)--(22,39.0)--(23,40.29)--(24,41.56)--(25,42.8)--(26,44.02)--(27,45.21)--(28,46.38)--(29,47.52)--(30,48.65)--(31,49.75)--(32,50.83)--(33,51.89)--(34,52.93)--(35,53.95)--(36,54.94)--(37,55.92)--(38,56.88)--(39,57.82)--(40,58.74)--(41,59.64)--(42,60.52)--(43,61.38)--(44,62.23)--(45,63.05)--(46,63.86)--(47,64.65)--(48,65.42)--(49,66.18)--(50,66.91)--(51,67.63)--(52,68.33)--(53,69.01)--(54,69.67)--(55,70.32)--(56,70.95)--(57,71.56)--(58,72.16)--(59,72.73)--(60,73.29)--(61,73.83)--(62,74.36)--(63,74.87)--(64,75.36)--(65,75.83)--(66,76.28)--(67,76.72)--(68,77.14)--(69,77.54)--(70,77.93)--(71,78.29)--(72,78.64)--(73,78.98)--(74,79.29)--(75,79.59)--(76,79.86)--(77,80.12)--(78,80.37)--(79,80.59)--(80,80.8)--(81,80.99)--(82,81.16)--(83,81.31)--(84,81.44)--(85,81.56)--(86,81.65)--(87,81.73)--(88,81.79)--(89,81.83)--(90,81.85)--(91,81.85)--(92,81.83)--(93,81.79)--(94,81.73)--(95,81.65)--(96,81.55)--(97,81.44)--(98,81.3)--(99,81.14)--(100,80.96)--(101,80.76)--(102,80.54)--(103,80.29)--(104,80.03)--(105,79.74)--(106,79.44)--(107,79.11)--(108,78.76)--(109,78.38)--(110,77.99)--(111,77.57)--(112,77.12)--(113,76.66)--(114,76.17)--(115,75.66)--(116,75.12)--(117,74.56)--(118,73.97)--(119,73.36)--(120,72.73)--(121,72.07)--(122,71.38)--(123,70.67)--(124,69.93)--(125,69.16)--(126,68.37)--(127,67.55)--(128,66.71)--(129,65.83)--(130,64.93)--(131,64.0)--(132,63.05)--(133,62.06)--(134,61.05)--(135,60.01)--(136,58.94)--(137,57.84)--(138,56.71)--(139,55.55)--(140,54.36)--(141,53.14)--(142,51.89)--(143,50.61)--(144,49.3)--(145,47.95)--(146,46.58)--(147,45.17)--(148,43.73)--(149,42.25)--(150,40.73)--(151,39.17)--(152,37.58)--(153,35.94)--(154,34.26)--(155,32.52)--(156,30.74)--(157,28.91)--(158,27.02)--(159,25.07)--(160,23.05)--(161,20.96)--(162,18.79)--(163,16.53)--(164,14.17)--(165,11.7)--(166,9.1)--(167,6.33)--(168,3.34)--(169,0.0);
        \filldraw (91,81.85) circle (1);\filldraw (169,0) circle (1);
        \draw (160,90)--(180,90);\node at(180,90) [right] {$N_{26\times26,c}$};
        \draw[densely dashed] (160,80)--(180,80);\node at(180,80) [right] {$N_{25\times26,c}$};
        \draw[densely dotted] (160,70)--(180,70);\node at(180,70) [right] {$N_{25\times25,c}$};
    }
    It can be observed that for these three different $(m,n)$, the maximum value of $c$ yielding a non-zero enumeration is consistently 169.
    While $N_{m\times n, c}$ exhibits similar growth trends with increasing $c$ across all cases, their final values (i.e., the enumeration of the maximum independent set) differ.

    Notably, when both $m$ and $n$ are odd, we have the trivial case $N_{m\times n, c}=1$, corresponding to the unique maximum independent set where vertices form a lattice with step size 2 that completely tiles the graph, as illustrated:
    \ctikzpic{0.6}{
        \draw (0,0) grid (4,4);
        \draw (0,3)--(1,4); \draw (0,2)--(2,4); \draw (0,1)--(3,4); \draw (0,0)--(4,4); \draw (1,0)--(4,3); \draw (2,0)--(4,2); \draw (3,0)--(4,1);
        \draw (4,3)--(3,4); \draw (4,2)--(2,4); \draw (4,1)--(1,4); \draw (4,0)--(0,4); \draw (3,0)--(0,3); \draw (2,0)--(0,2); \draw (1,0)--(0,1);
        \node at(0,0) {\wking};\filldraw[fill=white] (0,1) circle (0.1);\node at(0,2) {\wking};\filldraw[fill=white] (0,3) circle (0.1);\node at(0,4) {\wking};
        \filldraw[fill=white] (1,0) circle (0.1);\filldraw[fill=white] (1,1) circle (0.1);\filldraw[fill=white] (1,2) circle (0.1);\filldraw[fill=white] (1,3) circle (0.1);\filldraw[fill=white] (1,4) circle (0.1);
        \node at(2,0) {\wking};\filldraw[fill=white] (2,1) circle (0.1);\node at(2,2) {\wking};\filldraw[fill=white] (2,3) circle (0.1);\node at(2,4) {\wking};
        \filldraw[fill=white] (3,0) circle (0.1);\filldraw[fill=white] (3,1) circle (0.1);\filldraw[fill=white] (3,2) circle (0.1);\filldraw[fill=white] (3,3) circle (0.1);\filldraw[fill=white] (3,4) circle (0.1);
        \node at(4,0) {\wking};\filldraw[fill=white] (4,1) circle (0.1);\node at(4,2) {\wking};\filldraw[fill=white] (4,3) circle (0.1);\node at(4,4) {\wking};
    }
    And when one of them is even (wlog, let the height $n$ be even), the maximum independent set can be viewed as allowing the kings to ``slide'' vertically like beads on an abacus (with a gap of height $1$, the gray areas in the following figure) based on the both-odd cases.
    \begin{center}
        \tikzpic{0.6}{
            \filldraw[fill=gray!20] (-1.5,-0.5) rectangle (4.5,6.5);

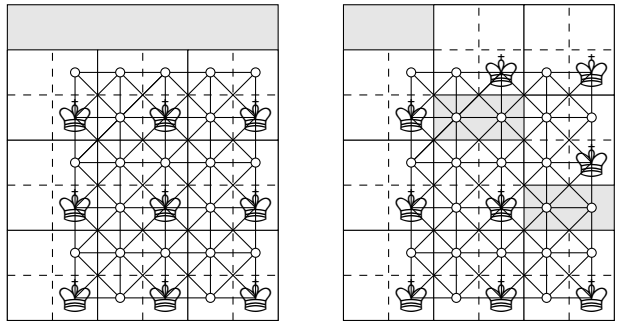
            \drawOiv{-0.5}{0.5};\drawOiv{-0.5}{2.5};\drawOiv{-0.5}{4.5};
            \drawOiv{1.5}{0.5};\drawOiv{1.5}{2.5};\drawOiv{1.5}{4.5};
            \drawOiv{3.5}{0.5};\drawOiv{3.5}{2.5};\drawOiv{3.5}{4.5};

            \draw (0,0) grid (4,5);
            \draw (0,4)--(1,5);\draw (0,3)--(2,5);\draw (0,3)--(2,5); \draw (0,2)--(3,5); \draw (0,1)--(4,5); \draw (0,0)--(4,4); \draw (1,0)--(4,3); \draw (2,0)--(4,2); \draw (3,0)--(4,1);
            \draw (4,4)--(3,5);\draw (4,3)--(2,5); \draw (4,2)--(1,5); \draw (4,1)--(0,5); \draw (4,0)--(0,4); \draw (3,0)--(0,3); \draw (2,0)--(0,2); \draw (1,0)--(0,1);

            \node at(0,0) {\wking};\filldraw[fill=white] (0,1) circle (0.1);\node at(0,2) {\wking};\filldraw[fill=white] (0,3) circle (0.1);\node at(0,4) {\wking};\filldraw[fill=white] (0,5) circle (0.1);
            \filldraw[fill=white] (1,0) circle (0.1);\filldraw[fill=white] (1,1) circle (0.1);\filldraw[fill=white] (1,2) circle (0.1);\filldraw[fill=white] (1,3) circle (0.1);\filldraw[fill=white] (1,4) circle (0.1);\filldraw[fill=white] (1,5) circle (0.1);
            \node at(2,0) {\wking};\filldraw[fill=white] (2,1) circle (0.1);\node at(2,2) {\wking};\filldraw[fill=white] (2,3) circle (0.1);\node at(2,4) {\wking};\filldraw[fill=white] (2,5) circle (0.1);
            \filldraw[fill=white] (3,0) circle (0.1);\filldraw[fill=white] (3,1) circle (0.1);\filldraw[fill=white] (3,2) circle (0.1);\filldraw[fill=white] (3,3) circle (0.1);\filldraw[fill=white] (3,4) circle (0.1);\filldraw[fill=white] (3,5) circle (0.1);
            \node at(4,0) {\wking};\filldraw[fill=white] (4,1) circle (0.1);\node at(4,2) {\wking};\filldraw[fill=white] (4,3) circle (0.1);\node at(4,4) {\wking};\filldraw[fill=white] (4,5) circle (0.1);
        }~~~~
        \tikzpic{0.6}{
            \filldraw[fill=gray!20] (-1.5,-0.5) rectangle (4.5,6.5);

            \drawOiv{-0.5}{0.5};\drawOiv{-0.5}{2.5};\drawOiv{-0.5}{4.5};
            \drawOiv{1.5}{0.5};\drawOiv{1.5}{2.5};\drawOiv{1.5}{5.5};
            \drawOiv{3.5}{0.5};\drawOiv{3.5}{3.5};\drawOiv{3.5}{5.5};

            \draw (0,0) grid (4,5);
            \draw (0,4)--(1,5);\draw (0,3)--(2,5);\draw (0,3)--(2,5); \draw (0,2)--(3,5); \draw (0,1)--(4,5); \draw (0,0)--(4,4); \draw (1,0)--(4,3); \draw (2,0)--(4,2); \draw (3,0)--(4,1);
            \draw (4,4)--(3,5);\draw (4,3)--(2,5); \draw (4,2)--(1,5); \draw (4,1)--(0,5); \draw (4,0)--(0,4); \draw (3,0)--(0,3); \draw (2,0)--(0,2); \draw (1,0)--(0,1);
            \node at(0,0) {\wking};\filldraw[fill=white] (0,1) circle (0.1);\node at(0,2) {\wking};\filldraw[fill=white] (0,3) circle (0.1);\node at(0,4) {\wking};\filldraw[fill=white] (0,5) circle (0.1);
            \filldraw[fill=white] (1,0) circle (0.1);\filldraw[fill=white] (1,1) circle (0.1);\filldraw[fill=white] (1,2) circle (0.1);\filldraw[fill=white] (1,3) circle (0.1);\filldraw[fill=white] (1,4) circle (0.1);\filldraw[fill=white] (1,5) circle (0.1);
            \node at(2,0) {\wking};\filldraw[fill=white] (2,1) circle (0.1);\node at(2,2) {\wking};\filldraw[fill=white] (2,3) circle (0.1);\filldraw[fill=white] (2,4) circle (0.1);\node at(2,5) {\wking};
            \filldraw[fill=white] (3,0) circle (0.1);\filldraw[fill=white] (3,1) circle (0.1);\filldraw[fill=white] (3,2) circle (0.1);\filldraw[fill=white] (3,3) circle (0.1);\filldraw[fill=white] (3,4) circle (0.1);\filldraw[fill=white] (3,5) circle (0.1);
            \node at(4,0) {\wking};\filldraw[fill=white] (4,1) circle (0.1);\filldraw[fill=white] (4,2) circle (0.1);\node at(4,3) {\wking};\filldraw[fill=white] (4,4) circle (0.1);\node at(4,5) {\wking};
        }
    \end{center}
    It is straightforward to observe that in this configuration there are exactly $\frac{m+1}{2}$ columns of kings, each with $\frac{n}{2}+1$ possible positions for the gaps,
    so we have the enumeration of maximum independent set
    \[
        N_{m\times n,\lfloor\frac{m+1}{2}\rfloor\lfloor\frac{n+1}{2}\rfloor}=
        \begin{cases}
            1, & 2\nmid m,2\nmid n; \\
            \left(\frac{n}{2}+1\right)^{\frac{m+1}{2}}, & 2\nmid m,2\mid n; \\
            \left(\frac{m}{2}+1\right)^{\frac{n+1}{2}}, & 2\mid m,2\nmid n; \\
            \ldots, & 2\mid m,2\nmid n.
        \end{cases}
    \]

    Only in cases where both $m$ and $n$ are even does the enumeration of maximum independent sets lack a simple closed-form formula,
    while simultaneously allowing greater degrees of freedom for king moves, so that the enumeration is significantly larger than in the adjacent cases.
    \cite{MaxInd, MaxInd2} provide further discussion on maximum independent sets in king's graphs.

    \subsection{Maximum entropy density}

    We now examine the vertex number that maximizes the independent set enumeration.
    The ratio of this vertex number to the total number of vertices corresponds to the \textit{maximum entropy density} in the hard-core model,
    which is the particle density that maximizes the total entropy of the system.
    Denote by $c_{m\times n}$ the value of $c$ that maximizes $N_{m\times n, c}$ for given $m$ and $n$, and when there are two adjacent maximum points, take the smaller one.
    And denote by
    \[
        \hat{N}_{m\times n}\coloneq N_{m \times n, c_{m\times n}}
    \]
    the maximum enumeration.

    The following table shows the data of $c_{m\times n}$ in the case of $m, n\leq 27$:

    \begin{center}
        \tiny
        \begin{tabular}{r|l}
            \hline
            ${}_m{}^n$ & 1~~\phantom{0}2~~\phantom{0}3~~\phantom{0}4~~\phantom{0}5~~\phantom{0}6~~\phantom{0}7~~\phantom{0}8~~\phantom{0}9~~10~~11~~12~~13~~14~~15~~16~~17~~18~~19~~20~~21~~22~~23~~24~~25~~26 \\
            \hline
            1 & 0~~\phantom{0}1~~\phantom{0}1~~\phantom{0}1~~\phantom{0}2~~\phantom{0}2~~\phantom{0}2~~\phantom{0}2~~\phantom{0}3~~\phantom{0}3~~\phantom{0}3~~\phantom{0}4~~\phantom{0}4~~\phantom{0}4~~\phantom{0}4~~\phantom{0}5~~\phantom{0}5~~\phantom{0}5~~\phantom{0}5~~\phantom{0}6~~\phantom{0}6~~\phantom{0}6~~\phantom{0}7~~\phantom{0}7~~\phantom{0}7~~\phantom{0}7 \\
            2 & 1~~\phantom{0}1~~\phantom{0}1~~\phantom{0}2~~\phantom{0}2~~\phantom{0}2~~\phantom{0}3~~\phantom{0}3~~\phantom{0}3~~\phantom{0}4~~\phantom{0}4~~\phantom{0}4~~\phantom{0}5~~\phantom{0}5~~\phantom{0}5~~\phantom{0}6~~\phantom{0}6~~\phantom{0}6~~\phantom{0}7~~\phantom{0}7~~\phantom{0}7~~\phantom{0}8~~\phantom{0}8~~\phantom{0}8~~\phantom{0}9~~\phantom{0}9 \\
            3 & 1~~\phantom{0}1~~\phantom{0}2~~\phantom{0}2~~\phantom{0}3~~\phantom{0}4~~\phantom{0}4~~\phantom{0}5~~\phantom{0}5~~\phantom{0}6~~\phantom{0}6~~\phantom{0}7~~\phantom{0}7~~\phantom{0}8~~\phantom{0}8~~\phantom{0}9~~\phantom{0}9~~10~~10~~11~~11~~12~~13~~13~~14~~14 \\
            4 & 1~~\phantom{0}2~~\phantom{0}2~~\phantom{0}3~~\phantom{0}4~~\phantom{0}4~~\phantom{0}5~~\phantom{0}6~~\phantom{0}6~~\phantom{0}7~~\phantom{0}7~~\phantom{0}8~~\phantom{0}9~~\phantom{0}9~~10~~11~~11~~12~~13~~13~~14~~14~~15~~16~~16~~17 \\
            5 & 2~~\phantom{0}2~~\phantom{0}3~~\phantom{0}4~~\phantom{0}4~~\phantom{0}5~~\phantom{0}6~~\phantom{0}7~~\phantom{0}8~~\phantom{0}8~~\phantom{0}9~~10~~11~~12~~12~~13~~14~~15~~16~~16~~17~~18~~19~~20~~20~~21 \\
            6 & 2~~\phantom{0}2~~\phantom{0}4~~\phantom{0}4~~\phantom{0}5~~\phantom{0}6~~\phantom{0}7~~\phantom{0}8~~\phantom{0}9~~10~~11~~12~~13~~14~~14~~15~~16~~17~~18~~19~~20~~21~~22~~23~~24~~24 \\
            7 & 2~~\phantom{0}3~~\phantom{0}4~~\phantom{0}5~~\phantom{0}6~~\phantom{0}7~~\phantom{0}8~~\phantom{0}9~~10~~11~~12~~13~~15~~16~~17~~18~~19~~20~~21~~22~~23~~24~~25~~26~~27~~28 \\
            8 & 2~~\phantom{0}3~~\phantom{0}5~~\phantom{0}6~~\phantom{0}7~~\phantom{0}8~~\phantom{0}9~~10~~12~~13~~14~~15~~16~~18~~19~~20~~21~~22~~24~~25~~26~~27~~28~~29~~31~~32 \\
            9 & 3~~\phantom{0}3~~\phantom{0}5~~\phantom{0}6~~\phantom{0}8~~\phantom{0}9~~10~~12~~13~~14~~16~~17~~18~~20~~21~~22~~24~~25~~26~~28~~29~~30~~32~~33~~34~~36 \\
            10 & 3~~\phantom{0}4~~\phantom{0}6~~\phantom{0}7~~\phantom{0}8~~10~~11~~13~~14~~16~~17~~19~~20~~22~~23~~25~~26~~27~~29~~30~~32~~33~~35~~36~~38~~39 \\
            11 & 3~~\phantom{0}4~~\phantom{0}6~~\phantom{0}7~~\phantom{0}9~~11~~12~~14~~16~~17~~19~~20~~22~~24~~25~~27~~28~~30~~32~~33~~35~~36~~38~~40~~41~~43 \\
            12 & 4~~\phantom{0}4~~\phantom{0}7~~\phantom{0}8~~10~~12~~13~~15~~17~~19~~20~~22~~24~~26~~27~~29~~31~~33~~34~~36~~38~~40~~41~~43~~45~~47 \\
            13 & 4~~\phantom{0}5~~\phantom{0}7~~\phantom{0}9~~11~~13~~15~~16~~18~~20~~22~~24~~26~~28~~30~~31~~33~~35~~37~~39~~41~~43~~45~~46~~48~~50 \\
            14 & 4~~\phantom{0}5~~\phantom{0}8~~\phantom{0}9~~12~~14~~16~~18~~20~~22~~24~~26~~28~~30~~32~~34~~36~~38~~40~~42~~44~~46~~48~~50~~52~~54 \\
            15 & 4~~\phantom{0}5~~\phantom{0}8~~10~~12~~14~~17~~19~~21~~23~~25~~27~~30~~32~~34~~36~~38~~40~~42~~45~~47~~49~~51~~53~~55~~57 \\
            16 & 5~~\phantom{0}6~~\phantom{0}9~~11~~13~~15~~18~~20~~22~~25~~27~~29~~31~~34~~36~~38~~41~~43~~45~~47~~50~~52~~54~~57~~59~~61 \\
            17 & 5~~\phantom{0}6~~\phantom{0}9~~11~~14~~16~~19~~21~~24~~26~~28~~31~~33~~36~~38~~41~~43~~45~~48~~50~~53~~55~~58~~60~~62~~65 \\
            18 & 5~~\phantom{0}6~~10~~12~~15~~17~~20~~22~~25~~27~~30~~33~~35~~38~~40~~43~~45~~48~~51~~53~~56~~58~~61~~63~~66~~68 \\
            19 & 5~~\phantom{0}7~~10~~13~~16~~18~~21~~24~~26~~29~~32~~34~~37~~40~~42~~45~~48~~51~~53~~56~~59~~61~~64~~67~~69~~72 \\
            20 & 6~~\phantom{0}7~~11~~13~~16~~19~~22~~25~~28~~30~~33~~36~~39~~42~~45~~47~~50~~53~~56~~59~~62~~64~~67~~70~~73~~76 \\
            21 & 6~~\phantom{0}7~~11~~14~~17~~20~~23~~26~~29~~32~~35~~38~~41~~44~~47~~50~~53~~56~~59~~62~~65~~68~~71~~73~~76~~79 \\
            22 & 6~~\phantom{0}8~~12~~14~~18~~21~~24~~27~~30~~33~~36~~40~~43~~46~~49~~52~~55~~58~~61~~64~~68~~71~~74~~77~~80~~83 \\
            23 & 7~~\phantom{0}8~~13~~15~~19~~22~~25~~28~~32~~35~~38~~41~~45~~48~~51~~54~~58~~61~~64~~67~~71~~74~~77~~80~~83~~87 \\
            24 & 7~~\phantom{0}8~~13~~16~~20~~23~~26~~29~~33~~36~~40~~43~~46~~50~~53~~57~~60~~63~~67~~70~~73~~77~~80~~84~~87~~90 \\
            25 & 7~~\phantom{0}9~~14~~16~~20~~24~~27~~31~~34~~38~~41~~45~~48~~52~~55~~59~~62~~66~~69~~73~~76~~80~~83~~87~~91~~94 \\
            26 & 7~~\phantom{0}9~~14~~17~~21~~24~~28~~32~~36~~39~~43~~47~~50~~54~~57~~61~~65~~68~~72~~76~~79~~83~~87~~90~~94~~98 \\
            27 & 8~~\phantom{0}9~~15~~18~~22~~25~~29~~33~~37~~41~~44~~48~~52~~56~~60~~63~~67~~71~~75~~79~~82~~86~~90~~94~~98~~-- \\
            \hline
        \end{tabular}
        \normalsize
    \end{center}

    Noticed that $c_{m\times n}$ is relatively small but constrained to be integer-valued, we employ quadratic polynomial interpolation to obtain an approximation $\bar{c}_{m\times n}$ for better characterization of its variation pattern.
    Specifically, when $c_{m\times n}\geq 1$ (if not, take $\bar{c}_{m\times n}=c_{m\times n}$), we construct a quadratic polynomial $f$ satisfying
    \[
        \begin{cases}
            f(c_{m\times n}-1) &= \ln N_{m\times n,c_{m\times n}-1};\\
            f(c_{m\times n})   &= \ln N_{m\times n, c_{m\times n}};\\
            f(c_{m\times n}+1) &= \ln N_{m\times n,c_{m\times n}+1},
        \end{cases}
    \]
    then use its maximum point $\bar{c}_{m\times n}$ to take the place of $c_{m\times n}$.
    In accordance with our previously established patterns (the monotonicity of the finite differences), we have:
    \[
        c_{m\times n}=\lfloor \bar{c}_{m\times n}+ 0.5 \rfloor.
    \]
    Although such interpolation can also give the corresponding $\bar{N}_{m\times n}$, we still use the original $N_{m\times n}$,
    because when $m, n$ are large, the difference between $\ln\bar{N}_{m\times n}$ and $\ln N_{m\times n}$ can be ignored.

    We define the (interpolation adjusted) maximum entropy density as
    \[
        \rho_{m\times n}\coloneq\frac{\bar{c}_{m\times n}}{mn}.
    \]
    The following charts record the cases of $m=n, 2n,3n$, which clearly demonstrates the growth trend of $\bar{\rho}_{m\times n}$:

    % n axis
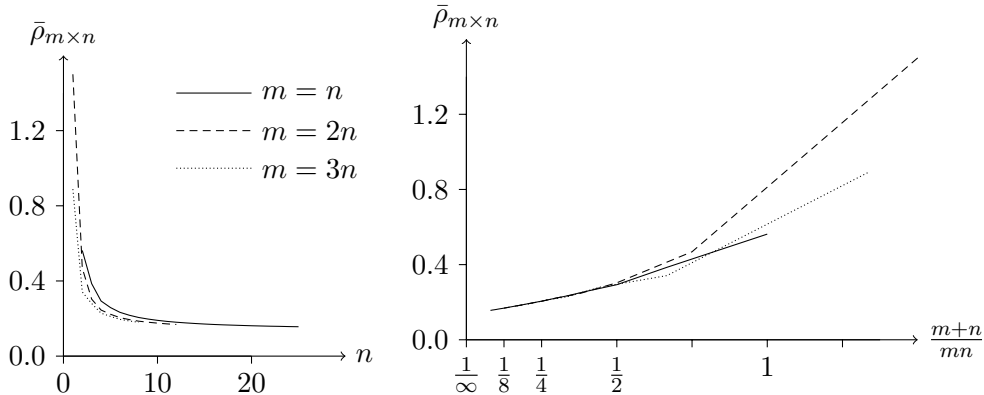
    \begin{center}
        \tikzpic{1}{
            \draw[->](0,0)--(3.75,0);\node at(3.75,0)[right]{$n$};
            \draw[step=1.25] (0,-0.1)grid(2.5,0);\node at(0,-0.1)[below]{0};\node at(1.25,-0.1)[below]{10};\node at(2.5,-0.1)[below]{20};
            \draw[->](0,0)--(0,4);\node at(0,4) [above] {$\bar{\rho}_{m\times n}$};
            \draw (-0.1,0)grid(0,3.5);\node at(-0.1,0)[left]{0.0};\node at(-0.1,1)[left]{0.4};\node at(-0.1,2)[left]{0.8};\node at(-0.1,3)[left]{1.2};
            \draw                 (0.250,1.406)--(0.375,0.964)--(0.500,0.733)--(0.625,0.646)--(0.750,0.583)--(0.875,0.544)--(1.000,0.514)--(1.125,0.493)--(1.250,0.476)--(1.375,0.462)--(1.500,0.451)--(1.625,0.442)--(1.750,0.434)--(1.875,0.427)--(2.000,0.421)--(2.125,0.416)--(2.250,0.412)--(2.375,0.408)--(2.500,0.404)--(2.625,0.401)--(2.750,0.398)--(2.875,0.396)--(3.000,0.393)--(3.125,0.391);
            \draw[densely dashed] (0.125,3.750)--(0.250,1.171)--(0.375,0.758)--(0.500,0.613)--(0.625,0.557)--(0.750,0.512)--(0.875,0.486)--(1.000,0.465)--(1.125,0.451)--(1.250,0.439)--(1.375,0.429)--(1.500,0.421);
            \draw[densely dotted] (0.125,2.222)--(0.250,0.854)--(0.375,0.700)--(0.500,0.571)--(0.625,0.529)--(0.750,0.489)--(0.875,0.468)--(1.000,0.450);

            \draw (1.5,3.5)--(2.5,3.5);\node at(2.5,3.5) [right] {$m=n$};
            \draw[densely dashed] (1.5,3.0)--(2.5,3.0);\node at(2.5,3.0) [right] {$m=2n$};
            \draw[densely dotted] (1.5,2.5)--(2.5,2.5);\node at(2.5,2.5) [right] {$m=3n$};
        }
        % (m+n)/mn axis
        \tikzpic{1}{
            \draw[->](0,0)--(6,0);\node at(6,0)[right]{$\frac{m+n}{mn}$};
            \draw (0,-0.1)grid(5.5,0);\draw (0.5,-0.1)--(0.5,0);\node at(0,-0.1)[below]{$\frac{1}{\infty}$};\node at(0.5,-0.1)[below]{$\frac{1}{8}$};\node at(1,-0.1)[below]{$\frac{1}{4}$};\node at(2,-0.1)[below]{$\frac{1}{2}$};\node at(4,-0.1)[below]{1};
            \draw[->](0,0)--(0,4);\node at(0,4) [above] {$\bar{\rho}_{m\times n}$};
            \draw (-0.1,0)grid(0,3.5);\node at(-0.1,0)[left]{0.0};\node at(-0.1,1)[left]{0.4};\node at(-0.1,2)[left]{0.8};\node at(-0.1,3)[left]{1.2};
            \draw                 (4.000,1.406)--(2.666,0.964)--(2.000,0.733)--(1.600,0.646)--(1.333,0.583)--(1.142,0.544)--(1.000,0.514)--(0.888,0.493)--(0.800,0.476)--(0.727,0.462)--(0.666,0.451)--(0.615,0.442)--(0.571,0.434)--(0.533,0.427)--(0.500,0.421)--(0.470,0.416)--(0.444,0.412)--(0.421,0.408)--(0.400,0.404)--(0.381,0.401)--(0.363,0.398)--(0.347,0.396)--(0.333,0.393)--(0.320,0.391);
            \draw[densely dashed] (6.000,3.750)--(3.000,1.171)--(2.000,0.758)--(1.500,0.613)--(1.200,0.557)--(1.000,0.512)--(0.857,0.486)--(0.750,0.465)--(0.666,0.451)--(0.600,0.439)--(0.545,0.429)--(0.500,0.421);
            \draw[densely dotted] (5.333,2.222)--(2.666,0.854)--(1.777,0.700)--(1.333,0.571)--(1.066,0.529)--(0.888,0.489)--(0.761,0.468)--(0.666,0.450);
        }
    \end{center}

    It can be observed that the trend of $\bar{\rho}_{m\times n}$ resemble that of $\kappa_{m\times n}$ in the previous chapter.
    Although a rigorous proof cannot be provided, we may hypothesize the existence of the limits:
    \begin{align*}
        \bar{\rho}     & \coloneq \bar{\rho}_{m\times n}; \\
        k_{\bar{\rho}} & \coloneq (\bar{\rho}_{m\times n}-\bar{\rho})(m+n).
    \end{align*}
    Consequently, when $m,n\rightarrow\infty$, we can approximate $\rho$ as:
    \[
        \bar{\rho}_{m\times n} \sim \bar{\rho}+\frac{k_{\bar{\rho}}}{m+n}, \\
    \]
    and furthermore,
    \[
        c_{m\times n} \sim \left(\bar{\rho}+\frac{k_{\bar{\rho}}}{m+n}\right)mn. \\
    \]

    Unfortunately, because we need to rely on interpolation to make up for the continuity of the data, the relevant data of K does not have very good characteristics, so we can only give an approximate estimate rather than the upper and lower bounds.
    Based on available data, we estimate (after the $\pm$ sign is the approximate error range):
    \begin{align*}
        \bar{\rho}       &\approx 0.1366741 \pm 0.0002374; \\
        k_{\bar{\rho}}   &\approx 0.102081 \pm 0.00275.
    \end{align*}
    Similarly, for
    \[
        \hat{\kappa}_{m\times n}\coloneq\frac{\ln\hat{N}_{m\times n}}{mn},
    \]
    we have:
    \begin{center}
        \tikzpic{1}{
            \draw[->](0,0)--(3.75,0);\node at(3.75,0)[right]{$n$};
            \draw[step=1.25] (0,-0.1)grid(3.5,0);\node at(0,-0.1)[below]{0};\node at(1.25,-0.1)[below]{10};\node at(2.5,-0.1)[below]{20};
            \draw[->](0,0)--(0,4);\node at(0,4) [above] {$\hat{\kappa}_{m\times n}$};
            \draw (-0.1,0)grid(0,3.5);\node at(-0.1,0)[left]{0.29};\node at(-0.1,1)[left]{0.31};\node at(-0.1,2)[left]{0.33};\node at(-0.1,3)[left]{0.35};
            \draw (0.250,2.828)--(0.375,0.903)--(0.500,0.942)--(0.625,0.688)--(0.750,0.832)--(0.875,0.823)--(1.000,0.787)--(1.125,0.794)--(1.250,0.772)--(1.375,0.758)--(1.500,0.741)--(1.625,0.724)--(1.750,0.707)--(1.875,0.693)--(2.000,0.678)--(2.125,0.665)--(2.250,0.652)--(2.375,0.639)--(2.500,0.627)--(2.625,0.616)--(2.750,0.605)--(2.875,0.595)--(3.000,0.585)--(3.125,0.576)--(3.250,0.568);
            \draw[densely dashed] (0.125,2.828)--(0.250,1.030)--(0.375,0.837)--(0.500,0.872)--(0.625,0.908)--(0.750,0.877)--(0.875,0.851)--(1.000,0.822)--(1.125,0.793)--(1.250,0.761)--(1.375,0.735)--(1.500,0.713)--(1.625,0.691)--(1.750,0.671)--(1.875,0.652)--(2.000,0.635)--(2.125,0.619);
            \draw[densely dotted] (0.125,3.810)--(0.250,0.870)--(0.375,1.326)--(0.500,1.061)--(0.625,1.036)--(0.750,0.953)--(0.875,0.912)--(1.000,0.853)--(1.125,0.818)--(1.250,0.780)--(1.375,0.748)--(1.500,0.718)--(1.625,0.692);

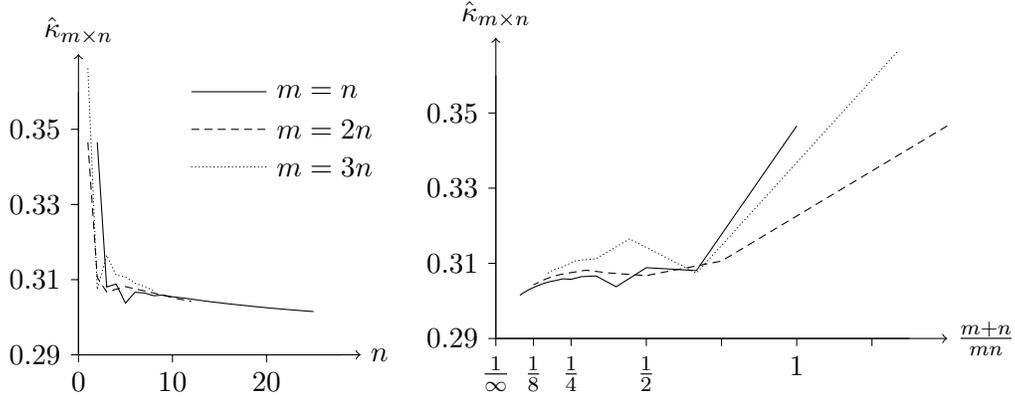
            \draw (1.5,3.5)--(2.5,3.5);\node at(2.5,3.5) [right] {$m=n$};
            \draw[densely dashed] (1.5,3.0)--(2.5,3.0);\node at(2.5,3.0) [right] {$m=2n$};
            \draw[densely dotted] (1.5,2.5)--(2.5,2.5);\node at(2.5,2.5) [right] {$m=3n$};
        }
        % (m+n)/mn axis
        \tikzpic{1}{
            \draw[->](0,0)--(6,0);\node at(6,0)[right]{$\frac{m+n}{mn}$};
            \draw (0,-0.1)grid(5.5,0);\draw (0.5,-0.1)--(0.5,0);\node at(0,-0.1)[below]{$\frac{1}{\infty}$};\node at(0.5,-0.1)[below]{$\frac{1}{8}$};\node at(1,-0.1)[below]{$\frac{1}{4}$};\node at(2,-0.1)[below]{$\frac{1}{2}$};\node at(4,-0.1)[below]{1};
            \draw[->](0,0)--(0,4);\node at(0,4) [above] {$\hat{\kappa}_{m\times n}$};
            \draw (-0.1,0)grid(0,3.5);\node at(-0.1,0)[left]{0.29};\node at(-0.1,1)[left]{0.31};\node at(-0.1,2)[left]{0.33};\node at(-0.1,3)[left]{0.35};
            \draw (4.000,2.828)--(2.666,0.903)--(2.000,0.942)--(1.600,0.688)--(1.333,0.832)--(1.142,0.823)--(1.000,0.787)--(0.888,0.794)--(0.800,0.772)--(0.727,0.758)--(0.666,0.741)--(0.615,0.724)--(0.571,0.707)--(0.533,0.693)--(0.500,0.678)--(0.470,0.665)--(0.444,0.652)--(0.421,0.639)--(0.400,0.627)--(0.381,0.616)--(0.363,0.605)--(0.347,0.595)--(0.333,0.585)--(0.320,0.576)--(0.307,0.568);
            \draw[densely dashed] (6.000,2.828)--(3.000,1.030)--(2.000,0.837)--(1.500,0.872)--(1.200,0.908)--(1.000,0.877)--(0.857,0.851)--(0.750,0.822)--(0.666,0.793)--(0.600,0.761)--(0.545,0.735)--(0.500,0.713)--(0.461,0.691)--(0.428,0.671)--(0.400,0.652)--(0.375,0.635)--(0.352,0.619);
            \draw[densely dotted] (5.333,3.810)--(2.666,0.870)--(1.777,1.326)--(1.333,1.061)--(1.066,1.036)--(0.888,0.953)--(0.761,0.912)--(0.666,0.853)--(0.592,0.818)--(0.533,0.780)--(0.484,0.748)--(0.444,0.718)--(0.410,0.692);
        }
    \end{center}
    In fact, when $m,n\rightarrow0$, both $\kappa_{m\times n}$ and $\hat{\kappa}_{m\times n}$ converge to the same constant, i.e.,
    \begin{proposition}
        \[
            \lim_{m,n\rightarrow\infty}\hat{\kappa}_{m\times n}=\kappa.
        \]
    \end{proposition}
    \begin{proof}
        Clearly for any $m,n$, $\hat{N}_{m\times n}\leq N_{m\times n}$, which implies
        \[
            \lim_{m,n\rightarrow\infty}\hat{\kappa}_{m\times n} \coloneq\lim_{m,n\rightarrow\infty}\frac{\ln \hat{N}_{m\times n}}{mn} \leq\kappa \coloneq \lim_{m,n\rightarrow\infty}\frac{\ln N_{m\times n}}{mn}.
        \]
        At this time, there are only $\lfloor\frac{m+1}{2}\rfloor\lfloor\frac{n+1}{2}\rfloor+1$ values of $c$ that let $N_{m\times n,c}$ not be zero,
        and the sum of these $N_{m\times n,c}$ is $N_{m\times n}$.
        So its maximum value under the variance of $c$
        \[
            \hat{N}_{m\times n}\geq \frac{N_{m\times n}}{\lfloor\frac{m+1}{2}\rfloor\lfloor\frac{n+1}{2}\rfloor+1} \geq \frac{4N_{m\times n}}{(m+1)(n+1)}.
        \]
        Thus,
        \begin{align*}
            & \lim_{m,n\rightarrow\infty}\frac{\ln \hat{N}_{m\times n}}{mn}\\
            \geq & \lim_{m,n\rightarrow\infty}\frac{\ln \frac{4N_{m\times n}}{(m+1)(n+1)}}{mn}\\
            = & \lim_{m,n\rightarrow\infty}\frac{\ln N_{m\times n}-\ln (m+1) -\ln (n+1) +\ln 4}{mn}\\
            =& \kappa -\frac{\ln (m+1) +\ln (n+1) -\ln 4}{mn}.
        \end{align*}
        And when $m,n\rightarrow\infty$,$\frac{\ln (m+1) +\ln (n+1) -\ln 4}{mn}\rightarrow0$.
    \end{proof}

    Therefore, we only need to assume that there is the limit
    \[
        k_{\hat{\kappa}} \coloneq (\hat{\kappa}_{m\times n}-\hat{\kappa})(m+n),
    \]
    so that for when $m,n\rightarrow\infty$,
    \[
        \bar{\kappa}_{m\times n} \sim \kappa+\frac{k_{\bar{\kappa}}}{m+n},
    \]
    and furthermore,
    \[
        \ln N_{m\times n, c_{m\times n}} \sim (\bar{\kappa}+\frac{k_{\bar{\kappa}}}{m+n})mn,
    \]

    % kappa_exa=0.294640767...
    Similar to $\bar{\rho}$ and $k_{\bar{\rho}}$, our estimation of $k_{\bar{\kappa}}$ and $\kappa$ here is:
    \begin{align*}
        k_{\bar{\kappa}} &\approx 0.085610\pm0.018115; \\
        \kappa           &\approx 0.294601\pm0.001516.
    \end{align*}

    The absolute error of $\kappa$ in this estimation is about $4.0 \times10^{-5}$, which is obviously not as accurate as the estimation given by $\kappa_{m\times n}$.
    This may be because the range of computable $m, n$ in this case is smaller, or because of the model itself.

    \section{Vertex-number weighted enumeration}

    In this chapter, we examine the weighted enumeration of independent sets where the weight corresponds to the vertex number.
    We denote this weighted enumeration for an $m\times n$ king graph as $W(m,n)$.

    It is straightforward to observe that, without altering the positions of the kings, if one of the kings is colored black (while all others remain white), the number of possible colorings always equals the total number of kings.
    Consequently, $W(m,n)$ equals to the enumeration of non-attacking arrangements of two-colored kings on an $m\times n$ chessboard where exactly one king is black.
    The following figure shows an example with 5 vertices (kings):

    \begin{center}
        \tikzpic{0.6}{
            \draw (0,0) grid (3,4);
            \draw (0,3)--(1,4); \draw (0,2)--(2,4); \draw (0,1)--(3,4); \draw (0,0)--(3,3); \draw (1,0)--(3,2); \draw (2,0)--(3,1);
            \draw (3,3)--(2,4); \draw (3,2)--(1,4); \draw (3,1)--(0,4); \draw (3,0)--(0,3); \draw (2,0)--(0,2); \draw (1,0)--(0,1);
            \node at(0,0) {\wking};\filldraw[fill=white] (0,1) circle (0.1);\filldraw[fill=white] (0,2) circle (0.1);\node at(0,3) {\wking};\filldraw[fill=white] (0,4) circle (0.1);
            \filldraw[fill=white] (1,0) circle (0.1);\filldraw[fill=white] (1,1) circle (0.1);\filldraw[fill=white] (1,2) circle (0.1);\filldraw[fill=white] (1,3) circle (0.1);\filldraw[fill=white] (1,4) circle (0.1);
            \filldraw[fill=white] (2,0) circle (0.1);\filldraw[fill=white] (2,1) circle (0.1);\node at(2,2) {\wking};\filldraw[fill=white] (2,3) circle (0.1);\filldraw[fill=white] (2,4) circle (0.1);
            \node at(3,0) {\wking};\filldraw[fill=white] (3,1) circle (0.1);\filldraw[fill=white] (3,2) circle (0.1);\filldraw[fill=white] (3,3) circle (0.1);\node at(3,4) {\bking};
        }
        \tikzpic{0.6}{
            \draw (0,0) grid (3,4);
            \draw (0,3)--(1,4); \draw (0,2)--(2,4); \draw (0,1)--(3,4); \draw (0,0)--(3,3); \draw (1,0)--(3,2); \draw (2,0)--(3,1);
            \draw (3,3)--(2,4); \draw (3,2)--(1,4); \draw (3,1)--(0,4); \draw (3,0)--(0,3); \draw (2,0)--(0,2); \draw (1,0)--(0,1);
            \node at(0,0) {\wking};\filldraw[fill=white] (0,1) circle (0.1);\filldraw[fill=white] (0,2) circle (0.1);\node at(0,3) {\bking};\filldraw[fill=white] (0,4) circle (0.1);
            \filldraw[fill=white] (1,0) circle (0.1);\filldraw[fill=white] (1,1) circle (0.1);\filldraw[fill=white] (1,2) circle (0.1);\filldraw[fill=white] (1,3) circle (0.1);\filldraw[fill=white] (1,4) circle (0.1);
            \filldraw[fill=white] (2,0) circle (0.1);\filldraw[fill=white] (2,1) circle (0.1);\node at(2,2) {\wking};\filldraw[fill=white] (2,3) circle (0.1);\filldraw[fill=white] (2,4) circle (0.1);
            \node at(3,0) {\wking};\filldraw[fill=white] (3,1) circle (0.1);\filldraw[fill=white] (3,2) circle (0.1);\filldraw[fill=white] (3,3) circle (0.1);\node at(3,4) {\wking};
        }
        \tikzpic{0.6}{
            \draw (0,0) grid (3,4);
            \draw (0,3)--(1,4); \draw (0,2)--(2,4); \draw (0,1)--(3,4); \draw (0,0)--(3,3); \draw (1,0)--(3,2); \draw (2,0)--(3,1);
            \draw (3,3)--(2,4); \draw (3,2)--(1,4); \draw (3,1)--(0,4); \draw (3,0)--(0,3); \draw (2,0)--(0,2); \draw (1,0)--(0,1);
            \node at(0,0) {\wking};\filldraw[fill=white] (0,1) circle (0.1);\filldraw[fill=white] (0,2) circle (0.1);\node at(0,3) {\wking};\filldraw[fill=white] (0,4) circle (0.1);
            \filldraw[fill=white] (1,0) circle (0.1);\filldraw[fill=white] (1,1) circle (0.1);\filldraw[fill=white] (1,2) circle (0.1);\filldraw[fill=white] (1,3) circle (0.1);\filldraw[fill=white] (1,4) circle (0.1);
            \filldraw[fill=white] (2,0) circle (0.1);\filldraw[fill=white] (2,1) circle (0.1);\node at(2,2) {\bking};\filldraw[fill=white] (2,3) circle (0.1);\filldraw[fill=white] (2,4) circle (0.1);
            \node at(3,0) {\wking};\filldraw[fill=white] (3,1) circle (0.1);\filldraw[fill=white] (3,2) circle (0.1);\filldraw[fill=white] (3,3) circle (0.1);\node at(3,4) {\wking};
        }
        \tikzpic{0.6}{
            \draw (0,0) grid (3,4);
            \draw (0,3)--(1,4); \draw (0,2)--(2,4); \draw (0,1)--(3,4); \draw (0,0)--(3,3); \draw (1,0)--(3,2); \draw (2,0)--(3,1);
            \draw (3,3)--(2,4); \draw (3,2)--(1,4); \draw (3,1)--(0,4); \draw (3,0)--(0,3); \draw (2,0)--(0,2); \draw (1,0)--(0,1);
            \node at(0,0) {\bking};\filldraw[fill=white] (0,1) circle (0.1);\filldraw[fill=white] (0,2) circle (0.1);\node at(0,3) {\wking};\filldraw[fill=white] (0,4) circle (0.1);
            \filldraw[fill=white] (1,0) circle (0.1);\filldraw[fill=white] (1,1) circle (0.1);\filldraw[fill=white] (1,2) circle (0.1);\filldraw[fill=white] (1,3) circle (0.1);\filldraw[fill=white] (1,4) circle (0.1);
            \filldraw[fill=white] (2,0) circle (0.1);\filldraw[fill=white] (2,1) circle (0.1);\node at(2,2) {\wking};\filldraw[fill=white] (2,3) circle (0.1);\filldraw[fill=white] (2,4) circle (0.1);
            \node at(3,0) {\wking};\filldraw[fill=white] (3,1) circle (0.1);\filldraw[fill=white] (3,2) circle (0.1);\filldraw[fill=white] (3,3) circle (0.1);\node at(3,4) {\wking};
        }
        \tikzpic{0.6}{
            \draw (0,0) grid (3,4);
            \draw (0,3)--(1,4); \draw (0,2)--(2,4); \draw (0,1)--(3,4); \draw (0,0)--(3,3); \draw (1,0)--(3,2); \draw (2,0)--(3,1);
            \draw (3,3)--(2,4); \draw (3,2)--(1,4); \draw (3,1)--(0,4); \draw (3,0)--(0,3); \draw (2,0)--(0,2); \draw (1,0)--(0,1);
            \node at(0,0) {\wking};\filldraw[fill=white] (0,1) circle (0.1);\filldraw[fill=white] (0,2) circle (0.1);\node at(0,3) {\wking};\filldraw[fill=white] (0,4) circle (0.1);
            \filldraw[fill=white] (1,0) circle (0.1);\filldraw[fill=white] (1,1) circle (0.1);\filldraw[fill=white] (1,2) circle (0.1);\filldraw[fill=white] (1,3) circle (0.1);\filldraw[fill=white] (1,4) circle (0.1);
            \filldraw[fill=white] (2,0) circle (0.1);\filldraw[fill=white] (2,1) circle (0.1);\node at(2,2) {\wking};\filldraw[fill=white] (2,3) circle (0.1);\filldraw[fill=white] (2,4) circle (0.1);
            \node at(3,0) {\bking};\filldraw[fill=white] (3,1) circle (0.1);\filldraw[fill=white] (3,2) circle (0.1);\filldraw[fill=white] (3,3) circle (0.1);\node at(3,4) {\wking};
        }
    \end{center}

    Therefore, the corresponding hard-core model is equivalent to adding a specific particle of different type on the basis of the original model.

    \subsection{Algorithms}

    To realize the coloring of a single king, we introduce $\sharp'$ and $1'$ as new horizontal symbols for Wang tiles, where the $'$ notation indicates that no king has been colored yet.
    The top-left Wang tile set on the chessboard adds $'$ to the horizontal characters, and it is removed when a chosen king (located at the bottom-right corner of a $\Oiv$) is colored.
    The Wang tile set at the top-left corner is:
    \[
        \varU_{00}=\wtile{}{}{}{\sharp'}+\wtile{}{}{}{1'},
    \]
    which adds $'$ to the horizontal symbols.
    And the Wang tile set at non-boundary positions is
    \[
        \varU =\wtile{}{}{}{}+\wtile{}{}{}{1}+\wtile{1}{}{1}{}+\wtile{1}{1}{}{}+\wtile{\sharp'}{}{}{\sharp'}+\wtile{\sharp'}{}{}{1'}+\wtile{1'}{}{1}{\sharp'}+\wtilek{1'}{1}{}{\sharp'}+\wtilebk{1'}{1}{}{\sharp},
    \]
    where the last tile $\wtilebk{1'}{1}{}{\sharp}$ means to color the king black in this position (which removed the $'$ notation), others just keep the state of the $'$ notation.

    Naturally, we can further reduce the size of the state tensor by applying the tile merging method described in Section 2.

    \subsection{Results}

    We have computed all cases where $m\leq 36,m + n \leq 75$.
    The table below presents the results for all cases with $m = 36, n\leq 39$:
    \tiny
    \begin{align*}
        n~~ & W(m,n)\\
        1~~ & 394905492\\
        2~~ & 1119872954208\\
        3~~ & 437795895219640704\\
        4~~ & 10488449457877727581896\\
        5~~ & 950690466861189052025808624\\
        6~~ & 40405648811419545702973459599252\\
        7~~ & 2446512847037847019074434950183648624\\
        8~~ & 119561815924055428415685514629635629063188\\
        9~~ & 6391107867299675456409258933225136908100573930\\
        10~~ & 320284968162368720355981263263997702188252565628932\\
        \ldots & \ldots\\
        \shortstack[l]{34\\\phantom{0}}~~ &
        \shortstack[l]{9755945165263136061019925608137791069134878080484545173499730156288649758612128302014855530428622806\\347366699073793437075743503809158105083659604887625716988891640}\\
        \shortstack[l]{35\\\phantom{0}}~~ &
        \shortstack[l]{4645553237750484588511649712884027312983436639444355453638385105641225111161128480986374641610149728\\79028391294721887610178415740262275039100006675874386217577314324236}\\
        \shortstack[l]{36\\\phantom{0}}~~ &
        \shortstack[l]{2210369641008972305016602927856884263689257319123943385072578006107086319342102051173050329060548821\\2045894603366383583546448335730654797700700199993957434658988856246436972}\\
        \shortstack[l]{37\\\phantom{0}}~~ &
        \shortstack[l]{1050920971705192128063573086017556059846116070548973949755107056821720782044044857192332312184925649\\238976584296182295004130314552293912705935328808089477794966752066019708763784}\\
        \shortstack[l]{38\\\phantom{0}}~~ &
        \shortstack[l]{4993094290959424444460539864800681414073576256421503669234504465425727527544781086334573559217562692\\7082807330070000930997886370824955799552494745338531907215632642200851690325561604}\\
        \shortstack[l]{39\\\phantom{0}}~~ &
        \shortstack[l]{2370716457457072140994798429008209568954369194026641755801790678933329335895788148466493103368090320\\339310982439989361269197827213635367070096218501842858565366526321085748516741054137890}
    \end{align*}
    \normalsize

    It can be observed that variation trend of $W_{m\times n}$ is very close to that of $N_{m\times n}$.
    Thus, we can similarly define
    \[
        \lambda_{m\times n} \coloneq \frac{\ln W_{m\times n}}{mn}.
    \]

    The chart below present some values of $\lambda_{m\times n}$ of cases $m=n,2n,3n$ derived from our result:
    \begin{center}
        % n axis
        \tikzpic{1}{
            \draw[->](0,0)--(5,0);\node at(5,0)[right]{$n$};
            \draw[step=1.25] (0,-0.1)grid(4.5,0);\node at(0,-0.1)[below]{0};\node at(1.25,-0.1)[below]{10};\node at(2.5,-0.1)[below]{20};\node at(3.75,-0.1)[below]{30};
            \draw[->](0,0)--(0,5.5);\node at(0,5.5) [above] {$\lambda_{m\times n}$};
            \draw (-0.1,0)grid(0,5.5);\node at(-0.1,0)[left]{0.28};\node at(-0.1,1)[left]{0.33};\node at(-0.1,2)[left]{0.38};\node at(-0.1,3)[left]{0.43};\node at(-0.1,4)[left]{0.48};\node at(-0.1,5)[left]{0.53};
            \draw (0.250,1.331)--(0.375,3.809)--(0.500,2.914)--(0.625,2.619)--(0.750,2.191)--(0.875,1.933)--(1.000,1.700)--(1.125,1.530)--(1.250,1.386)--(1.375,1.272)--(1.500,1.175)--(1.625,1.095)--(1.750,1.026)--(1.875,0.967)--(2.000,0.917)--(2.125,0.872)--(2.250,0.833)--(2.375,0.798)--(2.500,0.768)--(2.625,0.740)--(2.750,0.715)--(2.875,0.693)--(3.000,0.672)--(3.125,0.654)--(3.250,0.637)--(3.375,0.621)--(3.500,0.607)--(3.625,0.594)--(3.750,0.581)--(3.875,0.570)--(4.000,0.559)--(4.125,0.549)--(4.250,0.540)--(4.375,0.532)--(4.500,0.524);
            \draw[densely dashed] (0.125,1.331)--(0.250,3.064)--(0.375,3.141)--(0.500,2.337)--(0.625,1.978)--(0.750,1.646)--(0.875,1.438)--(1.000,1.268)--(1.125,1.143)--(1.250,1.041)--(1.375,0.960)--(1.500,0.893)--(1.625,0.837)--(1.750,0.790)--(1.875,0.750)--(2.000,0.715)--(2.125,0.685)--(2.250,0.659)--(2.375,0.635)--(2.500,0.615)--(2.625,0.596)--(2.750,0.579)--(2.875,0.564)--(3.000,0.551)--(3.125,0.538);
            \draw[densely dotted] (0.125,5.129)--(0.250,3.127)--(0.375,2.809)--(0.500,2.027)--(0.625,1.705)--(0.750,1.415)--(0.875,1.240)--(1.000,1.096)--(1.125,0.992)--(1.250,0.908)--(1.375,0.841)--(1.500,0.786)--(1.625,0.741)--(1.750,0.702)--(1.875,0.670)--(2.000,0.641)--(2.125,0.617)--(2.250,0.595);

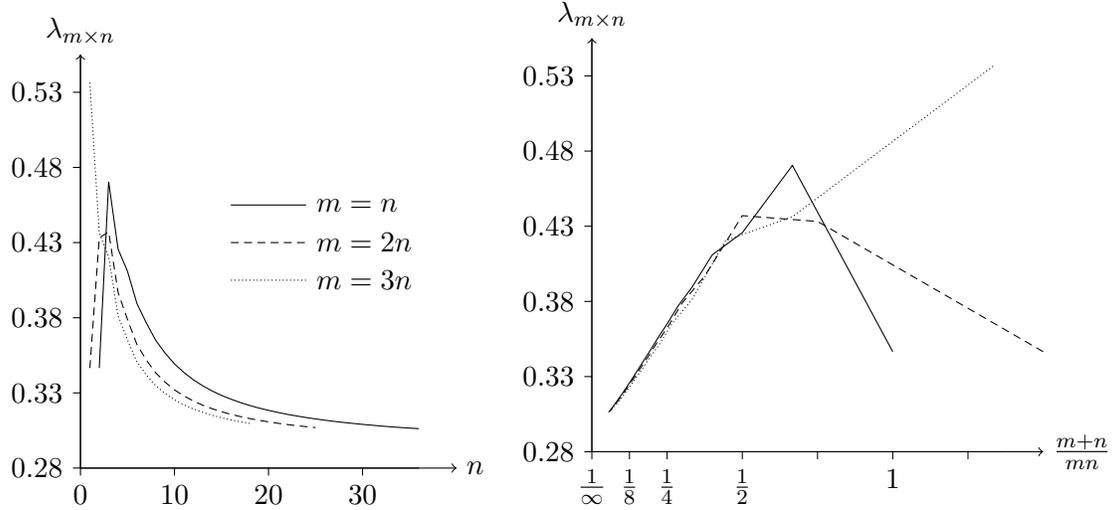
            \draw (2,3.5)--(3,3.5);\node at(3,3.5) [right] {$m=n$};
            \draw[densely dashed] (2,3.0)--(3,3.0);\node at(3,3.0) [right] {$m=2n$};
            \draw[densely dotted] (2,2.5)--(3,2.5);\node at(3,2.5) [right] {$m=3n$};
        }
        % (m+n)/mn axis
        \tikzpic{1}{
            \draw[->](0,0)--(6,0);\node at(6,0)[right]{$\frac{m+n}{mn}$};
            \draw (0,-0.1)grid(5.5,0);\draw (0.5,-0.1)--(0.5,0);\node at(0,-0.1)[below]{$\frac{1}{\infty}$};\node at(0.5,-0.1)[below]{$\frac{1}{8}$};\node at(1,-0.1)[below]{$\frac{1}{4}$};\node at(2,-0.1)[below]{$\frac{1}{2}$};\node at(4,-0.1)[below]{1};
            \draw[->](0,0)--(0,5.5);\node at(0,5.5) [above] {$\lambda_{m\times n}$};
            \draw (-0.1,0)grid(0,5.5);\node at(-0.1,0)[left]{0.28};\node at(-0.1,1)[left]{0.33};\node at(-0.1,2)[left]{0.38};\node at(-0.1,3)[left]{0.43};\node at(-0.1,4)[left]{0.48};\node at(-0.1,5)[left]{0.53};
            \draw (4.000,1.331)--(2.666,3.809)--(2.000,2.914)--(1.600,2.619)--(1.333,2.191)--(1.142,1.933)--(1.000,1.700)--(0.888,1.530)--(0.800,1.386)--(0.727,1.272)--(0.666,1.175)--(0.615,1.095)--(0.571,1.026)--(0.533,0.967)--(0.500,0.917)--(0.470,0.872)--(0.444,0.833)--(0.421,0.798)--(0.400,0.768)--(0.381,0.740)--(0.363,0.715)--(0.347,0.693)--(0.333,0.672)--(0.320,0.654)--(0.307,0.637)--(0.296,0.621)--(0.285,0.607)--(0.275,0.594)--(0.266,0.581)--(0.258,0.570)--(0.250,0.559)--(0.242,0.549)--(0.235,0.540)--(0.228,0.532)--(0.222,0.524);
            \draw[densely dashed] (6.000,1.331)--(3.000,3.064)--(2.000,3.141)--(1.500,2.337)--(1.200,1.978)--(1.000,1.646)--(0.857,1.438)--(0.750,1.268)--(0.666,1.143)--(0.600,1.041)--(0.545,0.960)--(0.500,0.893)--(0.461,0.837)--(0.428,0.790)--(0.400,0.750)--(0.375,0.715)--(0.352,0.685)--(0.333,0.659)--(0.315,0.635)--(0.300,0.615)--(0.285,0.596)--(0.272,0.579)--(0.260,0.564)--(0.250,0.551)--(0.240,0.538);
            \draw[densely dotted] (5.333,5.129)--(2.666,3.127)--(1.777,2.809)--(1.333,2.027)--(1.066,1.705)--(0.888,1.415)--(0.761,1.240)--(0.666,1.096)--(0.592,0.992)--(0.533,0.908)--(0.484,0.841)--(0.444,0.786)--(0.410,0.741)--(0.381,0.702)--(0.355,0.670)--(0.333,0.641)--(0.313,0.617)--(0.296,0.595);
        }
    \end{center}

    \subsection{Estimation of constants}

    The behavior of $\lambda_{m\times n}$ is also similar to that of $\kappa_{m\times n}$,
    and
    \begin{proposition}
        \[
            \lim_{m,n\rightarrow0} \lambda_{m\times n}=\kappa.
        \].
    \end{proposition}
    \begin{proof}
        This is because, when taking logarithms, their weights (number of tiles) become negligible compared to the tiling enumeration.

        Specifically, noted that the number of vertices is always not greater than $\frac{m+1}{2}\cdot\frac{n+1}{2}$,
        \begin{align*}
             \lambda_{m\times n} &= \frac{\ln W_{m\times n}}{mn}\\
                                &\leq \frac{\ln(\frac{m+1}{2}\cdot\frac{n+1}{2}\cdot N_{m\times n})}{mn}\\
                                &= \frac{\ln\frac{m+1}{2}+\ln\frac{n+1}{2} +\ln N_{m\times n})}{mn}\\
                                &=\kappa_{m\times n}+\frac{\ln(m+1)+\ln(n+1)-2\ln 2}{mn},
        \end{align*}
        and when $m,n\rightarrow\infty$,$\frac{\ln(m+1)+\ln(n+1)-2\ln 2}{mn}\rightarrow0$.
        Furthermore, excluding the empty vertex set, the number is always at least 1, so that when $m\geq 3$ or $n \geq 3$,
        there exists the independent set with at least two vertices, and
        \[
            W_{m\times n}\geq N_{m\times n}, \lambda_{m\times n}\geq\kappa_{m\times n}.
        \]
    \end{proof}
    Thus, to estimate $\lambda_{m\times n}$, it suffices to define
    \[
        k_\lambda \coloneq \lim_{m,n\rightarrow\infty} (\lambda_{m\times n}-\kappa)(mn),
    \]
    yielding when $m,n\rightarrow\infty$,
    \begin{align*}
        \lambda_{m\times n} & \sim \kappa+\frac{k_\lambda}{m+n};\\
        \ln W_{m \times n}  & \sim (\kappa+\frac{k_\lambda}{m+n})mn.
    \end{align*}

    Compared with the case without weighting, the regularity of the data at this time is poor, and the speed of numerical convergence is significantly slower.
    The estimation of $k_\lambda$ we get from the data of $m\leq 36,n\leq39$ is:

    \[
        0.1~4615 < k_\lambda \lessapprox 0.1~5145. \\
    \]
    It can be seen that the accuracy is significantly lower.
    And at this time, $\kappa$ can only give a lower bound close to it because of its poor regularity:
    % exact: kap=0.294640767816144918..
    \[
        0.2946407112 \lessapprox \kappa
    \]

    The absolute error of $\kappa$ in this estimation is about $5.7 \times10^{-7}$, which is obviously not as accurate as the estimation given by $\kappa_{m\times n}$.

    \bibliographystyle{unsrt}
    \bibliography{main}

\begin{thebibliography}{10}

\bibitem{tilEnum}
Kai Liang.
\newblock Solving tiling enumeration problems by tensor network contractions,
  2025.

\bibitem{Enum}
Richard~J. Mathar.
\newblock Tiling n x m rectangles with 1 x 1 and s x s squares.
\newblock {\em arXiv: Combinatorics}, 2016.

\bibitem{Enum2}
Nilsson Johan.
\newblock On counting the number of tilings of a rectangle with squares of size
  1 and 2.
\newblock {\em Journal of Integer Sequences}, 20, 2017.

\bibitem{lim2}
Steven~R. Finch.
\newblock Several constants arising in statistical mechanics.
\newblock {\em Annals of Combinatorics}, 3(2):323--335, 1999.

\bibitem{entropy}
S.~Forchhammer and J.~Justesen.
\newblock Entropy bounds for constrained two-dimensional fields.
\newblock {\em IEEE Transactions on Information Theory}, 45(1):118--127, 1999.

\bibitem{const}
R.~J. Baxter, I.~G. Enting, and S.~K. Tsang.
\newblock Hard-square lattice gas.
\newblock {\em Journal of Statistical Physics}, 22(4):465--489, 1980.

\bibitem{const2}
G., S., and Joyce.
\newblock On the hard-hexagon model and the theory of modular functions.
\newblock {\em Philosophical Transactions of the Royal Society A Mathematical
  Physical \& Engineering Sciences}, 1988.

\bibitem{lim}
Neil~J. Calkin and Herbert~S. Wilf.
\newblock The number of independent sets in a grid graph.
\newblock {\em Siam Journal on Discrete Mathematics}, 11(1):54--60, 1998.

\bibitem{entropy2}
Shmuel Friedland and Uri~N. Peled.
\newblock Theory of computation of multidimensional entropy with an application
  to the monomer-dimer problem.
\newblock {\em Advances in Applied Mathematics}, 2005.

\bibitem{MaxInd}
Tricia~Muldoon Brown.
\newblock Using words to construct and enumerate maximum nonattacking
  chessboard arrangements.
\newblock In {\em Southern Georgia Mathematics International Conference}, 2024.

\bibitem{MaxInd2}
Tricia~Muldoon Brown.
\newblock Maximum arrangements of nonattacking kings on the $2n\times 2n$
  chessboard, 2022.

\end{thebibliography}

\end{document}